\documentclass[12pt,a4paper]{amsart}
\usepackage{amsopn,amssymb,mathrsfs,color}
\newtheorem{thm}{Theorem}[section]
\newtheorem{cor}[thm]{Corollary}
\newtheorem{lem}[thm]{Lemma}
\newtheorem{prop}[thm]{Proposition}
\newtheorem{prob}[thm]{Problem}
\newtheorem{fact}[thm]{Fact}

\theoremstyle{definition}
\newtheorem{defn}[thm]{Definition}
\newtheorem{example}[thm]{Example}
\newtheorem{rem}[thm]{Remark}

\newcommand{\ball}[1]{\ensuremath{B_{#1}}}
\newcommand{\bidual}[1]{\ensuremath{{#1}^{**}}}

\newcommand{\cderiv}[2]{\ensuremath{#1}^{({#2})}}

\newcommand{\closure}[1]{\ensuremath{\overline{{#1}}}}
\newcommand{\continuum}{\ensuremath{\mathfrak{c}}}

\newcommand{\czero}{\ensuremath{c_0}}

\newcommand{\czerok}[1]{\ensuremath{c_0({#1})}}

\newcommand{\diam}[1]{\ensuremath{\sdiam{({#1})}}}
\newcommand{\dual}[1]{\ensuremath{{#1}^*}}

\newcommand{\Gdelta}{\ensuremath{G_\delta}}

\newcommand{\lpk}[2]{\ensuremath{\ell_{#1}({#2})}}

\newcommand{\mapping}[3]{\ensuremath{{#1}:{#2}\longrightarrow{#3}}}

\newcommand{\nat}{\mathbb{N}}
\newcommand{\norm}[1]{\ensuremath{\|{#1}\|}}
\newcommand{\normdot}{\ensuremath{\|\cdot\|}}

\newcommand{\oneton}[2]{\ensuremath{{#1}_1,\ldots,{#1}_{#2}}}

\newcommand{\pnorm}[2]{\ensuremath{\|{#1}\|_{#2}}}
\newcommand{\pnormdot}[1]{\ensuremath{\|\cdot\|_{#1}}}

\newcommand{\rat}{\mathbb{Q}}
\newcommand{\real}{\mathbb{R}}

\newcommand{\restrict}[1]{\ensuremath{\!\!\upharpoonright_{#1}}}
\newcommand{\setcomp}[2]{\ensuremath{\left\{{#1}\;:\;\,{#2}\right\}}}
\newcommand{\sig}[2]{\ensuremath{\sigma({#1},{#2})}}
\newcommand{\sph}[1]{\ensuremath{S_{#1}}}

\newcommand{\tri}{{\displaystyle |\kern-.9pt|\kern-.9pt|}}
\newcommand{\trinorm}[1]{\ensuremath{\tri{#1}\tri}}
\newcommand{\trinormdot}{\ensuremath{\tri\cdot\tri}}

\newcommand{\weakstar}{\ensuremath{w^*}}
\newcommand{\wone}{\ensuremath{\omega_1}}

\DeclareMathOperator{\card}{card} %
\DeclareMathOperator{\conv}{conv} %
\DeclareMathOperator{\sdiam}{diam} %
\DeclareMathOperator{\str}{st} %
\DeclareMathOperator{\supp}{supp} %

\newcommand{\st}{($*$)}

\begin{document}
\title{Strictly convex norms and topology}
\begin{abstract}
We introduce a new topological property called {\st} and the
corresponding class of topological spaces, which includes spaces
with $\Gdelta$-diagonals and Gruenhage spaces. Using {\st}, we
characterise those Banach spaces which admit equivalent strictly
convex norms, and give an internal topological characterisation of
those scattered compact spaces $K$, for which the dual Banach space
$\dual{C(K)}$ admits an equivalent strictly convex dual norm. We
establish some relationships between {\st} and
other topological concepts, and the position of several well-known
examples in this context. For instance, we show that
$\dual{C(\mathcal{K})}$ admits an equivalent strictly convex dual
norm, where $\mathcal{K}$ is Kunen's compact space. Also, under the
continuum hypothesis CH, we give an example of a compact scattered
non-Gruenhage space having {\st}.
\end{abstract}

\author{J.\ Orihuela, R.\ J.\ Smith and S.\ Troyanski}
\date{\today}
\thanks{Some of this research was conducted during several visits
of R.\ Smith to the University of Murcia, Spain, from 2007 to 2010,
and during a visit of S.\ Troyanski to University College Dublin, in
2010.} \thanks{J.\ Orihuela and S.\ Troyanski were supported by
MTM2008-05396/MTM Fondos Feder and Fundaci\'on S\'eneca 008848/PI/08
CARM. S.\ Troyanski was also supported by the Institute of
Mathematics and Informatics, Bulgarian Academy of Sciences,
Bulgarian National Fund for Scientific Research contract DO
02-360/2008.}  \maketitle

\section{Introduction}

Hereafter, all Banach spaces will be assumed real and, unless
explicitly stated otherwise, all topological spaces will be Hausdorff.
Throughout this paper we will be defining new norms on existing
Banach spaces. These new norms will always be equivalent to the given
canonical norms. Banach space notation and terminology is standard
throughout.

A norm $\normdot$ on a Banach space $X$ is said to be {\em strictly
convex} (or rotund) if, given $x,\,y \in X$ satisfying $\norm{x} =
\norm{y} = \norm{\frac{1}{2}(x+y)}$, we have $x=y$ \cite[p.\
404]{clarkson:36}. Geometrically, this means that the unit sphere
$\sph{X}$ of $X$ in this norm has no non-trivial line segments, or,
equivalently, every element of $\sph{X}$ is an extreme point of the
unit ball $\ball{X}$.

Clearly, there are many Banach spaces whose natural norms are not
strictly convex. However, by appealing to the linear and topological
properties of a given space, it is often possible to define a new norm
that is strictly convex. Changing the norm in this way is often called
{\em renorming}. In certain cases, we would
like the new norm to possess in addition some form of lower
semicontinuity. For instance, we may wish for a norm on a dual space
$\dual{X}$ to be $\weakstar$-lower semicontinuous, so that it is the
dual of some norm on $X$. Alternatively, we may like a norm on a
$C(K)$-space to be lower semicontinuous with respect to the topology
of pointwise convergence. Such additional requirements can make
norms much more difficult to construct, but they do bestow certain
benefits. For example, if $\dual{X}$ can be endowed with a strictly
convex dual norm then the predual norm on $X$ is
automatically G\^ateaux smooth, by virtue of \v{S}mulyan's Lemma,
cf.\ \cite[Theorem I.1.4]{dgz:93}.

Despite the natural and intuitive nature of strict convexity, the
question of whether a Banach space may be given such a norm turns
out to be rather difficult to answer in general. A number of
mathematicians have sought to establish more easily verifiable
sufficient conditions and necessary conditions for a space to admit
a strictly convex norm. Before outlining this
paper, we mention some of the contributions to this collective
endeavour. Specialists will realise that it is possible to endow
many (but not all) of the spaces below with norms sporting stronger
properties than strict convexity, but we prefer not to dwell on such
properties here. For a fuller discussion, we refer the reader to
\cite{dgz:93,st:10}.

In \cite[Theorem 9]{clarkson:36}, it is shown that every separable
Banach space admits a strictly convex norm. By following Clarkson's
proof, Day showed that if a Banach space $X$ is separable then
$\dual{X}$ admits a strictly convex dual norm \cite[Theorem
4]{day:55}. If $\Gamma$ is a set then $\czerok{\Gamma}$ admits a
strictly convex norm \cite[Theorem 10]{day:55} (see also
\cite[Definition II.7.2]{dgz:93}). On the other hand, if $\Gamma$
is uncountable then the space $\ell_\infty^c(\Gamma)$ of
countably supported bounded functions $\mapping{x}{\Gamma}{\real}$
with the supremum norm is simply too big to admit a strictly convex
norm \cite[Theorem 8]{day:55} (\cite[Theorem II.7.12]{dgz:93}).

Amir and Lindenstrauss showed that if $X$ is weakly compactly
generated (WCG) then both $X$ and $\dual{X}$ admit a strictly convex
norm and a strictly convex dual norm, respectively \cite[Theorem
3]{al:68}. These results rely on the fact that if a bounded linear
map $\mapping{T}{X}{Y}$ is injective and $Y$ admits a strictly
convex norm, then so does $X$. If $X$ is WCG then we can find such
maps on both $X$ and $\dual{X}$, where $Y = \czerok{\Gamma}$ for
some $\Gamma$. Then \cite[Theorem 10]{day:55} can be applied.

At the time, such a `linear transfer' into some $\czerok{\Gamma}$
was the only way of showing that spaces admitted strictly convex
norms. Moreover, $\ell_\infty^c(\Gamma)$, $\Gamma$ uncountable, was
the `smallest' space known not to admit a strictly convex norm. In
\cite{dl:73}, the authors construct an increasing transfinite
sequence $(X_\alpha)_{1 \leq \alpha < \wone}$ of spaces of Baire-1
functions on $[0,1]$, all admitting strictly convex norms, and none
admitting a bounded linear injective map into any $\czerok{\Gamma}$,
provided $\alpha \geq 2$. Moreover, by refining Day's argument
\cite[Theorem 8]{day:55}, they showed that the union $Y =
\bigcup_{\alpha < \wone} X_\alpha$ does not admit a strictly convex
norm, and that there is no bounded linear injective map from
$\ell_\infty^c([0,1])$ into $Y$.

The fact that the dual of every WCG space admits strictly convex
dual norm, with a necessarily G\^ateaux smooth predual norm,
prompted Lindenstrauss to conjecture that if $X$ admits a G\^ateaux
smooth norm then it must embed as a subspace of some WCG space
\cite{lind:72}. Mercourakis provided a negative answer to this
conjecture by showing that if $X$ is a {\em weakly countably
determined} (WCD) space, then both $X$ and $\dual{X}$ admit strictly
convex norms \cite[Theorems 4.6 and 4.8]{mercourakis:87}, by virtue
of linear transfers (although not into $\czerok{\Gamma}$ in
general).

Papers such as \cite{dl:73,mercourakis:87} suggest that there is no
simple way of characterising strict convexity in terms of linear
structure. Since then, the problem of classifying Banach spaces
admitting strictly convex norms has been approached from a more
topological perspective, and particular attention has been paid to
strictly convex dual norms and $C(K)$-spaces. Any Banach space $X$
embeds isometrically into $C(\ball{\dual{X}},\weakstar)$, and this
fact enables certain results about $C(K)$-spaces to be generalised
to all Banach spaces, by phrasing them in terms of the topological
structure of $(\ball{\dual{X}},\weakstar)$.

For example, if $\dual{X}$ admits a strictly convex {\em dual} norm
then $(\ball{\dual{X}},\weakstar)$ is fragmentable \cite[Theorem
1.1]{ribarska:92}. We can say that a topological space is {\em
fragmentable} if it admits, for each $n \in \nat$, an increasing
well ordered family of open subsets $(U_\xi)_{\xi < \lambda_n}$, with
the property that given distinct points $x$ and $y$, we can find some
$n_0$ and $\xi < \lambda_{n_0}$ such that $\{x,y\} \cap U_\xi$ is a
singleton \cite[Theorem 1.9]{ribarska:87}. The idea of point separation
features throughout this paper. Indeed, the notion of strict
convexity can be viewed as a form of point separation.

The necessity condition above is far from sufficient however. The
class of fragmentable spaces is very large and includes, for
instance, all scattered spaces. Recall that a topological space is
{\em scattered} if every non-empty subspace admits a relatively
isolated point. In the year before \cite{mercourakis:87} appeared,
Talagrand showed that the space $\dual{C(\wone+1)}$ does not admit a
strictly convex dual norm \cite[Th\'eor\`eme 3]{tal:86}, where
$\wone$ is the first uncountable ordinal considered in its
(scattered) order topology. On the other hand, the dual unit ball
$(\ball{\dual{C(\wone+1)}},\weakstar)$ is fragmentable \cite[Theorem
3.1]{ribarska:87}.

The next significant sufficiency condition we mention requires a
definition.
\begin{defn}\label{descriptive}
A compact space $K$ is {\em descriptive} if it admits a {\em
$\sigma$-isolated network}, that is to say, a family $\mathscr{N} =
\bigcup_{n=1}^\infty \mathscr{N}_n$ of subsets of $K$, satisfying
\begin{enumerate}
\item $N \cap \closure{\bigcup\mathscr{N}_n\setminus\{N\}}$ is empty
whenever $N \in \mathscr{N}_n$ and $n \in \nat$, and
\item if $x \in U \subseteq K$, where $U$ is open, then there exists
$n \in \nat$ and $N \in \mathscr{N}_n$ such that $x \in N \subseteq
U$.
\end{enumerate}
\end{defn}
This topological covering property arose out of the theory of
`generalised metric spaces' \cite{gru:84}. The class of descriptive
compact spaces is large. For example, if $X$ is WCD then
$(\ball{\dual{X}},\weakstar)$ is descriptive (\cite[Th\'eor\`eme
3.6]{tal:79} and \cite[Corollary 2.4]{raja:03}). In \cite[Theorem
3.3]{raja:03}, Raja showed that if $K$ is {\em descriptive} then
$\dual{C(K)}$ admits a strictly convex dual norm. This result can be
adapted to give a sufficient condition which applies to a wide class
of dual Banach spaces \cite[Theorem 1.3]{or:04}, including duals of
WCD spaces. We remark that a compact scattered space $K$ is descriptive
if and only if it is {\em $\sigma$-discrete}, that is, $K =
\bigcup_{n=1} D_n$, where each $D_n$ is discrete in its relative topology.
This fact follows from \cite[Lemma 2.2]{raja:03}.

Despite these advances, there is a very large gap between the class
of descriptive spaces and $\wone+1$ and the more general class of
fragmentable spaces. Some years prior to the publication of
\cite{raja:03}, Haydon constructed some strictly convex dual norms
on spaces of the form $\dual{C(K)}$, where the $K$ are 1-point
compactifications of certain trees in their interval topologies
\cite[Theorem 7.1]{haydon:99}. It turns out that some of these
spaces are not descriptive, so Haydon's sufficient condition is not
covered by Raja's umbrella.

In \cite[Theorem 6]{smith:06}, the second-named author generalised
Haydon's result by characterising those trees for which the
associated spaces $\dual{C(K)}$ admit strictly convex dual norms.
Later, in \cite{smith:09}, this order-theoretic characterisation was
reproved in internal, topological terms. To state this result we
need another definition.
\begin{defn}\label{gruenhage}
A compact space $K$ is called {\em Gruenhage} if there
exists a sequence $(\mathscr{U}_n)_{n=1}^\infty$ of families of open subsets of $K$,
and sets $R_n$, $n \geq 1$, with the property that
\begin{enumerate}
\item if $x,y \in K$ are distinct, then there exists $n \in \nat$ and $U \in \mathscr{U}_n$, such that $\{x,y\} \cap U$ is a singleton, and
\item $U \cap V = R_n$ whenever $U,V \in \mathscr{U}_n$ are distinct.
\end{enumerate}
\end{defn}
This definition is equivalent to the original \cite[p.\ 372]{gru:87}
(see \cite[Proposition 2]{smith:09}). Every descriptive compact
space is Gruenhage \cite[Corollary 4]{smith:09}.

\begin{thm}[{\cite[Theorems 7 and 16]{smith:09}}]\label{smithgru}
Let $K$ be compact. Then the following statements hold.
\begin{enumerate}
\item If $K$ is Gruenhage then $\dual{C(K)}$ admits a strictly convex
dual lattice norm.
\item If $K$ is the 1-point compactification of a tree and $\dual{C(K)}$
admits a strictly convex dual norm, then $K$ is Gruenhage.
\end{enumerate}
\end{thm}

Theorem \ref{smithgru} (1) can be adapted to give a sufficient
condition (\cite[Corollary 10]{smith:09}) which applies to class of
dual Banach spaces even wider than that covered by \cite[Theorem
1.3]{or:04}. There are other instances of necessity besides Theorem
\ref{smithgru} (2). For instance, if the Banach space $X$ has an
(uncountable) unconditional basis then $\dual{X}$ admits a strictly
convex dual norm if and only if $(\ball{\dual{X}},\weakstar)$ is
Gruenhage (equivalently, if $(\ball{\dual{X}},\weakstar)$ is
descriptive) \cite[Theorem 6]{st:09}. Despite some courageous
attempts, it was not possible to prove the converse implication of
Theorem \ref{smithgru} (1). Many of the results of this paper are
the product of efforts to resolve this difficulty.

This paper is organised as follows. In Section \ref{gensc}, we
introduce a generalisation of Gruenhage's property, labeled {\st}
(Definition \ref{star}), and use it to give a characterisation of
Banach spaces which admit a strictly convex norm satisfying some
additional lower semicontinuity property (Theorem \ref{sccharac}).
This characterisation attempts to topologise as much as possible the
geometric condition of strict convexity. In Section \ref{norm}, we
use {\st} to find an analogue of Theorem \ref{smithgru} which
applies to all scattered compact spaces (Theorem \ref{scattered}).
This class is significant in Banach space theory because $C(K)$ is
an {\em Asplund} space if and only if $K$ is scattered. In doing so,
we show that {\st} comes close to providing a complete topological
characterisation of those $K$, for which $\dual{C(K)}$ admits a
strictly convex dual norm. In Section \ref{topology}, we establish
some of the topological properties of {\st} and its position in the
wider context of covering properties, and provide some examples of
scattered compact spaces, some of which having {\st} and others not.
In particular, we give an example of a scattered non-Gruenhage
compact space having {\st} (Example \ref{starnon-gru}). Thus,
Theorem \ref{scattered} does not follow from previous results such
as Theorem \ref{smithgru}. Along the way, we answer an open question
concerning Kunen's compact space: specifically, we show that it is
Gruenhage (Proposition \ref{kunen-gru}). In several cases, including
Example \ref{starnon-gru}, we shall assume extra principles
independent of the usual axioms of set theory. Finally, in Section
\ref{problems}, we present some open problems stemming from this
study.

\section{A characterisation of strict convexity in Banach spaces}\label{gensc}

In this section, we provide a general characterisation of strictly
convex renormings in Banach spaces. Throughout this section, $X$
will be a Banach space (and occasionally a general topological
space) and $F \subseteq \dual{X}$ a norming subspace. Recall that
$\sig{X}{F}$ denotes the coarsest topology on $X$ with respect to
which every element of $F$ is continuous. We begin by presenting a
useful folklore result, together with a brief sketch proof.

\begin{prop}\label{conv-charac}
Let $F \subseteq \dual{X}$ be a norming subspace. Suppose that there
exists a sequence of $\sig{X}{F}$-lower semicontinuous convex
functions $\mapping{\varphi_n}{X}{[0,\infty)}$ such that given
distinct $x,y \in X$, we can find $n \in \nat$ satisfying
\begin{equation*}\tag{1}\label{convsc}
{\textstyle \varphi_n(\frac{1}{2}(x+y))
< \max\{\varphi(x),\varphi(y)\}.}
\end{equation*}
Then $X$ admits a $\sig{X}{F}$-lower semicontinuous
strictly convex norm $\trinormdot$. Instead, if $X$ is a Banach
lattice, (\ref{convsc}) holds whenever $x,y \in X_+$ are distinct,
and
\[
\varphi_n(x) \leq \varphi_n(y)
\]
whenever $|x| \leq |y|$ and $n \in \nat$, then $\trinormdot$ is a
$\sig{X}{F}$-lower semicontinuous strictly convex lattice norm.
\end{prop}

\begin{proof}
Let $\normdot$ denote the original norm on $X$. We define a new norm
by
\[
\trinorm{x}^2 = \sum_{n,q} c_{n,q} \pnorm{x}{n,q}^2
\]
where $\pnormdot{n,q}$ is the Minkowski functional of
\[
C_{n,q} = \setcomp{x \in X}{\varphi_n(x)^2 + \varphi_n(-x)^2 \leq q}
\]
whenever $q$ is a rational number satisfying $q > 2\varphi_n(0)^2$,
and where the constants $c_{n,q} > 0$ are chosen to ensure the
uniform convergence of the sum on bounded sets. By a standard
convexity argument (cf.\ \cite[Fact II.2.3]{dgz:93}), it can be
shown that if $\trinorm{x} = \trinorm{y} = \frac{1}{2}\trinorm{x+y}$
then $\varphi_n(x) = \varphi_n(y) = \varphi_n(\frac{1}{2}(x+y))$ for
all $n$, whence $x=y$ by hypothesis. If we adopt the lattice
hypotheses instead then clearly $\trinormdot$ is also a lattice
norm, and strictly convex on $X_+$. To see that the strict convexity
extends to all of $X$, let $x,y \in X$ and suppose that $\trinorm{x}
= \trinorm{y} = \frac{1}{2}\trinorm{x+y}$. Then
$\frac{1}{2}\trinorm{\,|x| + |y|\,} = \trinorm{x}$ as well, so
strict convexity on $X_+$ yields $|x| = |y|$. If we set $w =
\frac{1}{2}(x+y)$ then repeating the above gives us $|x| = |w|$. A
simple lattice argument (e.g.\ \cite[p.\ 749]{smith:09}) leads us to
conclude that $x=y$.
\end{proof}

Our characterisation adopts several ideas from \cite{ot:09,ot:09a}.
Recall that if $A$ is a subset of a locally convex space then an
open {\em slice} $U$ of $A$ is the intersection of $A$ with an open
half-space of $X$. The following proposition will be our main tool.

\begin{prop}\label{scslices}
Let $A$ be a bounded subset of $X$ and $\mathscr{U}$ a family of
non-empty $\sig{X}{F}$-open slices of $A$. Then there exists a
$\sig{X}{F}$-lower semicontinuous 1-Lipschitz convex function
$\varphi$ with the property that whenever $x,y \in A$, $\{x,y\} \cap
\bigcup\mathscr{U}$ is non-empty and
\[
\varphi(x) \;=\; \varphi(y) \;=\; {\textstyle
\varphi{(\frac{1}{2}(x+y))}},
\]
we have $x,y \in U$ for some $U \in \mathscr{U}$.
\end{prop}

Proposition \ref{scslices} is an immediate corollary of the next
result, dubbed the `Slice Localisation Theorem'.

\begin{thm}[{\cite[Theorem 3]{ot:09}}]\label{slice-localisation}
Let $A$ be a bounded subset of $X$ and $\mathscr{U}$ a family of
non-empty $\sig{X}{F}$-open slices of $A$. Then there is an
equivalent $\sigma(X,F)$-lower semicontinuous norm $\normdot$ such
that for every sequence $(x_n)_{n=1}^\infty \subseteq X$ and $x \in
A \cap \bigcup\mathscr{U}$, if
\[
2\norm{x}^2 + 2\norm{x_n}^2 - \norm{x+x_n}^2 \to 0,
\]
then there is a sequence of slices $(U_n)_{n=1}^\infty \subseteq
\mathscr{U}$ and $n_0 \in \nat$ such that
\begin{enumerate}
\item $x,x_n \in U_n$ whenever $n \geq n_0$ and $x_n \in A$;
\item for every $\delta > 0$ there is some $n_\delta \in \nat$ such
that
\[
x, x_n \in \closure{(\conv(A \cap{U_n})+\delta\ball{X})}^{\sigma(X,F)}
\]
for all $n \geq n_\delta$.
\end{enumerate}
\end{thm}

The Slice Localisation Theorem can be used to simplify the proofs of
network characterisations of Banach spaces which admit locally
uniformly rotund norms. To prove Proposition \ref{scslices}, all we
need to do apply Theorem \ref{slice-localisation} with $x_n = y$ for
all $n$. However, there is a more transparent proof of this
proposition which we provide for completeness.

Of key importance to the proof is the concept of $F$-distance,
introduced in \cite{ot:09a}. Let  $D \subseteq X$ be a non-empty, convex
bounded subset. Given $\xi \in \bidual{X}$, define
\begin{equation}\label{Fbidual}
\pnorm{\xi}{F} \;=\; \sup\setcomp{\xi(f)}{f \in \ball{\dual{X}} \cap
F}.
\end{equation}
It is clear that $\pnormdot{F}$ is $\sig{\bidual{X}}{F}$-lower
semicontinuous ($\sig{\bidual{X}}{F}$ being the only generally
non-Hausdorff topology mentioned in this paper). Now set
\[
\varphi(x) \;=\; \inf\setcomp{\pnorm{x-d}{F}}{d \in
\closure{D}^{\sig{\bidual{X}}{\dual{X}}}}.
\]

\begin{defn}\label{Fdistance}
Given a non-empty, convex bounded subset $D \subseteq X$, we
call $\varphi(x)$ the {\em $F$-distance} from $x \in X$ to $D$.
\end{defn}

We pass to the bidual of $X$ in order to control the lower
semicontinuity properties of $\varphi$. The notion of $F$-distance
has a number of useful properties which we list in the next lemma.

\begin{lem}\label{Fdistanceprop}
Let $\varphi(x)$ be the {\em $F$-distance} from $x \in X$ to $D$.
\begin{enumerate}
\item $\varphi$ is convex and 1-Lipschitz;
\item $\varphi$ is $\sig{X}{F}$-lower semicontinuous;
\item $\closure{D}^{\sig{X}{F}} = \varphi^{-1}(0)$.
\end{enumerate}
\end{lem}

Properties (1) and (2) are proved in \cite[Proposition 2.1]{ot:09a}
and the third is a straightforward exercise involving the
Hahn-Banach separation theorem. Now we can give our alternative
proof of Proposition \ref{scslices}.

\begin{proof}[Proof of Proposition \ref{scslices}]
For each $U \in \mathscr{U}$ and $x \in X$, define $\varphi_U(x)$ to
be the $F$-distance from $x$ to $(\conv A)\setminus U$. Since $A$ is
bounded, we can define another convex, $\sig{X}{F}$-lower
semicontinuous, 1-Lipschitz function by
\[
\varphi(x) \;=\; \sup\setcomp{\varphi_U(x)}{U \in \mathscr{U}}.
\]
Let $x,y \in A$ with $\{x,y\} \cap \bigcup\mathscr{U}$ non-empty and
suppose that
\[
{\textstyle \varphi(x) = \varphi(y) = \varphi(\frac{1}{2}(x+y))}.
\]
Without loss of generality, we can assume that $x \in U$ for some
$U \in \mathscr{U}$. Since $U \cap \closure{(\conv A)\setminus
U}^{\sig{X}{F}}$ is empty, we have $\varphi(x) \geq \varphi_U(x)
> 0$ by Lemma \ref{Fdistanceprop}, part (3). Pick $\varepsilon
> 0$ such that $\varphi(x) > 5\varepsilon^2$ and choose $V \in
\mathscr{U}$ with the property that
\[
{\textstyle \varphi(\frac{1}{2}(x+y))^2 <
\varphi_V(\frac{1}{2}(x+y))^2 + \varepsilon^2}.
\]
We have
\begin{eqnarray*}
0 &=& {\textstyle \frac{1}{2}(\varphi(x)^2 + \varphi(y)^2) -
\varphi(\frac{1}{2}(x+y))^2}\\
&>& {\textstyle \frac{1}{2}(\varphi_V(x)^2 + \varphi_V(y)^2) -
\varphi_V(\frac{1}{2}(x+y))^2 - \varepsilon^2}\\
&\geq& {\textstyle \frac{1}{2}(\varphi_V(x)^2 + \varphi_V(y)^2) -
\frac{1}{4}(\varphi_V(x) + \varphi_V(y))^2 - \varepsilon^2}\\
&=& {\textstyle \frac{1}{4}(\varphi_V(x) - \varphi_V(y))^2 -
\varepsilon^2}
\end{eqnarray*}
thus
\begin{equation}\label{varphiest}
|\varphi_V(x) - \varphi_V(y)| \;<\; 2\varepsilon.
\end{equation}
Since $\varphi_V$ is convex, we have
$\max\{\varphi_V(x),\varphi_V(y)\} \geq
\varphi_V(\frac{1}{2}(x+y))$. Together with (\ref{varphiest}), this
implies
\begin{eqnarray*}
\min\{\varphi_V(x),\varphi_V(y)\} &\geq&
\max\{\varphi_V(x),\varphi_V(y)\} - 2\varepsilon\\
&\geq& {\textstyle \varphi_V(\frac{1}{2}(x+y))} - 2\varepsilon\\
&\geq& {\textstyle \left( \varphi(\frac{1}{2}(x+y)) - \varepsilon^2
\right)^{\frac{1}{2}} - 2\varepsilon}\\
&>& 0.
\end{eqnarray*}
Therefore $\varphi_V(x),\varphi_V(y) > 0$. Since $x,y \in A$, we get
$x,y \in V$.
\end{proof}

Proposition \ref{scslices} motivates the introduction of the central
topological concept featuring in this paper.

\begin{defn}
\label{star} We say that a topological space $X$ has {\st}
if there exists a sequence $(\mathscr{U}_n)_{n=1}^\infty$ of
families of open subsets of $X$, with the property that given any
$x,y \in X$, there exists $n \in \nat$ such that
\begin{enumerate}
\item $\{x,y\} \cap \bigcup\mathscr{U}_n$ is non-empty, and
\item $\{x,y\} \cap U$ is at most a singleton for all $U \in \mathscr{U}_n$.
\end{enumerate}
\end{defn}

Any sequence $(\mathscr{U}_n)_{n=1}^\infty$ satisfying the
conditions of Definition \ref{star} will be called a
{\em{\st}-sequence} for $X$. In addition, if $X$ is locally
convex and $A \subseteq X$ then we say $A$ has {\em {\st} with slices}
if $A$ admits a {\st}-sequence $(\mathscr{U}_n)_{n=1}^\infty$, with
the property that every element of $\bigcup_{n=1}^\infty
\mathscr{U}_n$ is an open slice of $A$.

\begin{rem}\label{sliceencode}
It will be convenient to note that if $A \subseteq X$, then to say that
$(A,\sig{X}{F})$ has {\st} with slices is equivalent to there being
a family of subsets $G_n \subseteq (\sph{\dual{X}} \cap F) \times
\real$, $n \in \nat$, such that given distinct $x,y \in A$, we have
$n \in \nat$ satisfying
\begin{enumerate}
\item[(a)] $\max\{f(x),f(y)\} > \lambda$ for some $(f,\lambda) \in
G_n$, and
\item[(b)] $\min\{g(x),g(y)\} \leq \mu$ for every $(g,\mu) \in G_n$.
\end{enumerate}
\end{rem}

Our characterisation follows.

\begin{thm}
\label{sccharac} Let $F \subseteq
\dual{X}$ be a 1-norming subspace. Then the following are equivalent.
\begin{enumerate}
\item $X$ admits a $\sig{X}{F}$-lower semicontinuous
strictly convex norm;
\item $(X,\sig{X}{F})$ has {\st} with slices;
\item $(\sph{X},\sig{X}{F})$ has {\st} with slices;
\item there is a sequence of subsets $(X_n)_{n=1}^\infty$ of $X$, such
that
\[
\setcomp{(x,y) \in X^2}{x \neq y} \subseteq \bigcup_{n=1}^\infty
X_n^2
\]
and where each $(X_n,\sig{X}{F})$ has {\st} with slices.
\end{enumerate}
\end{thm}

\begin{proof}
(1) $\Rightarrow$ (2): let $\normdot$ be a
$\sig{X}{F}$-lower semicontinuous strictly convex norm on $X$. Then
$F$ is also $1$-norming for $\normdot$. Let
\[
G_q = (\sph{\dual{(X,\normdot)}} \cap F) \times \{q\}
\]
for each rational number $q > 0$. We verify that $(X,\sig{X}{F})$
has {\st} by showing that the $G_q$ satisfy (a) and (b) of Remark
\ref{sliceencode}. Given distinct $x,y \in X$, assume
that $\norm{x} \leq
\norm{y}$. The strict convexity of $\normdot$ tells us that
$\norm{\frac{1}{2}(x+y)} < \norm{y}$. Let rational $q$ satisfy
$\norm{\frac{1}{2}(x+y)} < q < \norm{y}$. Since $F$ is 1-norming for
$\normdot$, we know that $f(y) > q$ for a pair $(f,q) \in G_q$,
giving (a). Now suppose $g(y) > q$ for some $(g,q) \in G_q$.
Then certainly $g(x) \leq q$, else we would have
\[
{\textstyle q < \frac{1}{2}g(x+y) \leq \frac{1}{2}\norm{x+y},}
\]
which doesn't make any sense. This shows that (b) is also satisfied.

(2) $\Rightarrow$ (3) is trivial because {\st} with slices is
inherited by subspaces. (3) $\Rightarrow$ (2): if
$(\sph{X},\sig{X}{F})$ has {\st} with slices then we take sets
$G_n$, $n \in \nat$ that satisfy (a) and (b) of Remark
\ref{sliceencode}. We can assume that $G_n \subseteq (\sph{\dual{X}}
\cap F) \times (-1,1)$ for every $n$. Given rational $q,r
> 0$, set
\[
H_q = (\sph{\dual{X}} \cap F) \times \{q\} \qquad\text{and}\qquad
L_{n,q,r} = \setcomp{(f,q(\lambda+r))}{(f,\lambda) \in G_n}.
\]
We claim that the $H_q$ and $L_{n,q,r}$ verify that $(X,\sig{X}{F})$
has {\st}, using Remark \ref{sliceencode}.

Let $x,y \in X$ be distinct, with $\norm{x} \leq \norm{y}$. If
$\norm{x} < \norm{y}$ then we choose rational $q$ to satisfy
$\norm{x} < q < \norm{y}$. Since $F$ is 1-norming, it is easy to
check that (a) and (b) are fulfilled by $H_q$. Now suppose
$\norm{x} = \norm{y}$. We know that, with respect to $x/\norm{x}$
and $y/\norm{y}$, (a) and (b) are satisfied by some $G_n$. Without
loss of generality, assume $f(x) > \norm{x}\lambda$, where
$(f,\lambda) \in G_n$. Our argument depends on the sign of
$\lambda$. If $\lambda \geq 0$ then choose rational $q,r
> 0$ satisfying
\[
f(x) > \norm{x}(\lambda + r)\qquad\text{and}\qquad
\frac{\norm{x}}{1+r} < q < \norm{x}.
\]
The constants have been arranged to ensure
\begin{equation}\label{starapprox}
\mu(\norm{x} - q) < \norm{x}-q < qr \qquad\text{whenever $|\mu| <
1$.}
\end{equation}
We have $f(x) > \norm{x}(\lambda+r)
> q(\lambda+r)$. Now suppose that $g(x) > q(\mu+r)$, where $(g,\mu)
\in G_n$. Then
\[
g(x) > q(\mu + r) > \norm{x}\mu
\]
by equation (\ref{starapprox}) above. This means $g(x/\norm{x}) >
\mu$, whence $g(y/\norm{y}) \leq \mu$ by (b), giving $g(y) <
q(\mu+r)$. In summary, we have shown that (a) and (b) of Remark
\ref{sliceencode} are fulfilled by $L_{n,q,r}$. If instead $\lambda
< 0$, we choose $r < -\lambda$ as above and ensure that $q$
satisfies
\[
\norm{x} < q < \frac{\norm{x}}{1-r}.
\]
By arguing similarly, we get what we want.

(2) $\Rightarrow$ (4) follows easily by setting $X_n = X$. We finish by
proving (4) $\Rightarrow$ (1). By taking intersections with
$m\ball{X}$, $m \in \nat$, and reindexing if necessary, we can assume that each
$X_n$ is bounded. Let each $X_n$ have a {\st}-sequence
$(\mathscr{U}_{n,m})_{m=1}^\infty$, where each element of
$\bigcup_{m=1}^\infty \mathscr{U}_{n,m}$ is a (non-empty)
$\sig{X}{F}$-open slice of $X_n$. Let $\varphi_{n,m}$ denote the
convex function constructed by applying Proposition \ref{scslices}
to $X_n$ and the family $\mathscr{U}_{n,m}$. We have ensured that if
$x,y \in X$ are distinct then we can find $n$ and $m$ such that
$\varphi_{n,m}(\frac{1}{2}(x+y)) <
\max\{\varphi_{n,m}(x),\varphi_{n,m}(y)\}$. The rest follows from
Proposition \ref{conv-charac}.
\end{proof}

Note that Theorem \ref{sccharac} (1), (2) and (4) are also equivalent
when $F$ is simply a norming subspace, rather than a 1-norming subspace.
We end this section by giving an example to show that the reliance
on slices in the statement of Theorem \ref{sccharac} is necessary
in general.

\begin{example}\label{slicesnecessary}
Let $K$ be the product $\{0,1\}^{\wone}$, endowed with the lexicographic
order topology. According to
\cite[Example 1]{hjnr:00}, $C(K)$ admits a Kadec norm
$\normdot$ but no strictly convex norm. By the definition
of Kadec norms, the weak topology agrees with the norm topology on
$\sph{(C(K),\normdot)}$. In particular $(\sph{(C(K),\normdot)},w)$
is metrisable, meaning that it has a $\sigma$-discrete base and thus
has {\st} as well. However, since $\normdot$ cannot be strictly convex,
Theorem \ref{sccharac} implies that
$(\sph{(C(K),\normdot)},w)$ does not have {\st} {\em with slices}.
\end{example}

We conclude this section by giving a sufficient condition for
constructing strictly convex norms. Theorem \ref{merc-suff} below
can be applied to many spaces of significance to the theory, such as
the Mercourakis spaces $c_1(\Sigma^\prime \times \Gamma)$ (see
\cite[Section VI.6]{dgz:93}), Dashiell-Lindenstrauss spaces and
spaces of the form $\dual{C(K)}$, where $K$ is Gruenhage. The idea,
which goes back to the classical norm of Day for $c_0(\Gamma)$
\cite[Theorem 10]{day:55}, is to `glue together' strictly convex
norms on finite-dimensional spaces (which are readily available) to
obtain strictly convex norms on larger spaces. Elements of Theorem
\ref{merc-suff} can be found in \cite[Theorem 5]{fmz:10}. Before
giving the theorem, we state a simple fact.

\begin{fact}\label{convpmfact}
Let $\mapping{\xi}{[0,1]}{\real}$ be a function satisfying
$\xi(0)\xi(1) < 0$, and suppose that $\xi_+$ and $\xi_-$ are convex.
Then for every $\lambda\in(0,1)$, we have
\begin{equation}\label{convexpm}
\xi_{\pm}(\lambda) < \lambda\xi_{\pm}(1) + (1-\lambda)\xi_{\pm}(0).
\end{equation}
\end{fact}

\begin{proof}
Assume $\xi(0) > 0$. Since $\xi_\pm$ are convex and $\xi$ is
necessarily continuous, it is easy to see that there is a unique
interval $[a,b]$, where $0 < a \leq b < 1$, such that
\[
\xi(u) > \xi(v) = 0 > \xi(w)
\]
whenever $u \in [0,a)$, $v \in [a,b]$ and $w \in (b,1]$. If $\lambda
\in [a,b]$ then clearly equation (\ref{convexpm}) holds for $\xi_\pm$. Let
$\lambda < a$. Then $\xi_-(\lambda) = 0$ and, as $\xi_-(1) > 0$,
(\ref{convexpm}) holds for $\xi_-$. Since $\xi_+$ is convex, setting
$\mu = \lambda/a$ gives
\[
\xi_+(\lambda) \leq (1-\mu)\xi_+(0) + \mu\xi_+(a) = (1-\mu)\xi_+(0)
< (1-\lambda)\xi_+(0)
\]
so (\ref{convexpm}) holds for $\xi_+$. The proof for the case
$\lambda > b$ is similar.
\end{proof}

Clearly, if $\xi$ is linear then $\xi_\pm$ are convex. The same is
true if $\xi$ is positive and convex.

\begin{thm}\label{merc-suff}
Let $\mapping{\Theta_n}{X}{\lpk{\infty}{\Gamma_n}}$ be a sequence of
maps such that both functions $x \mapsto \Theta_{n,\pm}(x)(\gamma)$
are $\sig{X}{F}$-lower semicontinuous and convex for every $\gamma
\in \Gamma_n$ and $n \in \nat$.

Let us assume in addition that for all distinct $x,y\in X$, there
are $\lambda \in (0,1)$, $n\in \nat$ and a finite set $A \subseteq
\Gamma_n$, such that
\begin{equation*}\tag{1}\label{dis}
\Theta_n(x)\restrict{A}\, \neq \Theta_n(y)\restrict{A}, \mbox{ and}
\end{equation*}
\begin{equation*}\tag{2}\label{mid}
|\Theta_n(z)(\alpha)| > |\Theta_n(z)(\gamma)| \mbox{ whenever }
\alpha \in A \mbox{ and }\gamma \in \Gamma \setminus A,
\end{equation*}
where $z=\lambda x+ (1-\lambda)y$. Then $X$ admits a
$\sigma(X,F)$-lower semicontinuous strictly convex norm
$\trinormdot$.

Instead, if $X$ is a Banach lattice, $\Theta_{n,\pm}(x) \leq
\Theta_{n,\pm}(y)$ whenever $|x| \leq |y|$ and equations (\ref{dis})
and (\ref{mid}) apply to distinct $x,y \in X_+$, then $\trinormdot$
is a $\sigma(X,F)$-lower semicontinuous strictly convex lattice
norm.
\end{thm}

\begin{proof}
Since $\Theta_{n,\pm}(\cdot)(\gamma)$ are both convex and
$\sig{X}{F}$-lower semicontinuous, the same is true of
$|\Theta_n(\cdot)(\gamma)|$. Define
\[
\Theta_{n,0}(x)(\gamma)= \Theta^2_n(x)(\gamma)
\qquad\mbox{and}\qquad \Theta_{n,\pm 1}(x)(\gamma) =
\Theta_{n,\pm}(x)(\gamma).
\]
If $\Gamma = \bigcup_{n=1}^\infty \Gamma_n$, $u \in \lpk{\infty}{\Gamma}$
and $A \subseteq \Gamma$ is finite, set
\[
\varphi_A(u) = \sum_{\gamma \in A}u(\gamma)
\]
and put $\varphi_{A,n,i}= \varphi_A\circ\Theta_{n,i}$ for every
$n\in \nat$ and $i \in \{-1,0,1\}$. Certainly, each
$\varphi_{A,n,i}$ is $\sig{X}{F}$-lower semicontinuous, non-negative
and convex. Finally, let
\[
\psi_{m,n,i}(x)=\sup\setcomp{\varphi_{A,n,i}(x)}{A \subseteq
\Gamma_n \mbox{ has cardinality }m}.
\]
To finish the proof, we shall show that for every distinct pair $x,y
\in X$, there is $m,n\in \nat$ and $i \in \{-1,0,1\}$ such that
\begin{equation} \label{conv}
{\textstyle \psi_{m,n,i}(\frac{1}{2}(x+y)) < \max \{
\psi_{m,n,i}(x),\, \psi_{m,n,i}(y)\}}.
\end{equation}
holds. Then we can appeal to Proposition \ref{conv-charac}.

Take $\lambda \in (0,1),$ $n \in \nat$ and $A \subseteq \Gamma_n$
satisfying (\ref{dis}) and (\ref{mid}). We consider two cases. First
suppose that $\Theta_n(x)(\beta)\Theta_n(y)(\beta) < 0$ for some
$\beta \in A$. From (\ref{mid}) we know that $\Theta_n(z)(\beta)
\neq 0$. Assume for now that $\Theta_n(z)(\beta) > 0$ and define the
non-empty set
\[
B = \setcomp{\alpha \in A}{\Theta_{n,+}(z)(\alpha) > 0}.
\]
so that
\[
\Theta_{n,+}(z)(\alpha) > \Theta_{n,+}(z)(\gamma)
\]
for every $\alpha \in B$ and $\gamma \in \Gamma\setminus B$.
Therefore $\psi_{n,m,1}(z) = \sum_{\alpha \in
B}\Theta_n(z)(\alpha)$, where $m$ is the cardinality of $B$.
Applying Fact \ref{convpmfact} to $\xi(t) = \Theta_n(tx +
(1-t)y)(\beta)$, $t \in [0,1]$, we get
\[
\Theta_{n,+}(z)(\beta) < \lambda\Theta_{n,+}(x)(\beta) +
(1-\lambda)\Theta_{n,+}(y)(\beta)
\]
whence
\[
\psi_{n,m,1}(z) < \lambda\psi_{n,m,1}(x) +
(1-\lambda)\psi_{n,m,1}(y)
\]
from which (\ref{conv}) quickly follows for $i=1$, by convexity. If
$\Theta_n(z)(\beta) < 0$ then we argue similarly with $i=-1$.

Let's now consider the case
\begin{equation}\label{zeroprod}
\Theta_n(x)(\alpha)\Theta_n(y)(\alpha) \geq 0
\end{equation}
for all $\alpha \in A$. Let $m \in \nat$ be the cardinality of $A$.
Since $t \mapsto t^2$ is strictly convex, from condition (\ref{dis})
we have
\begin{eqnarray*}
\sum_{\alpha \in A}(\lambda\Theta_n(x)(\alpha) +
(1-\lambda)\Theta_n(y)(\alpha))^2
&<& \sum_{\alpha \in A}\lambda(\Theta_n(x)(\alpha))^2 + (1-\lambda)(\Theta_n(y)(\alpha))^2\\
&=& \lambda\varphi_{A,n,0}(x) + (1-\lambda)\varphi_{A,n,0}(y)\\
&\leq& \lambda\psi_{m,n,0}(x) + (1-\lambda)\psi_{m,n,0}(y)\\
&\leq& \max\{\psi_{m,n,0}(x),\, \psi_{m,n,0}(y)\}.
\end{eqnarray*}
Given the convexity of $|\Theta_n(\cdot)(\alpha)|$ and equation
(\ref{zeroprod}), we obtain
\[
|\Theta_n(z)(\alpha)| = |\Theta_n(\lambda x + (1-\lambda)y)(\alpha)|
\leq |\lambda\Theta_n(x)(\alpha) + (1-\lambda)\Theta_n(y)(\alpha)|.
\]
This and condition (\ref{mid}) imply
\[
\psi_{m,n,0}(z) \;=\; \varphi_{A,n,0}(z) \;\leq\; \sum_{\alpha \in
A}(\lambda\Theta_n(x)(\alpha) + (1-\lambda)\Theta_n(y)(\alpha))^2.
\]
Combining these inequalities we see that
\[
\psi_{m,n,0}(z)< \max\{\psi_{m,n,0}(x),
\psi_{m,n,0}(y)\}
\]
from which (\ref{conv}) follows for $i=0$, again by convexity. If we
adopt the lattice assumptions instead, then each $\psi_{n,m,i}$
satisfies the lattice assumptions in Proposition \ref{conv-charac}.
\end{proof}

In the first corollary below is a sufficient condition of
`Mercourakis type', which is formally more general than similar
conditions given in the literature.

\begin{cor}\label{merc}
Let $X$ be a subspace or sublattice of $\lpk{\infty}{\Gamma}$ and
suppose that there are subsets $\Gamma_n \subseteq \Gamma$, $n \in
\nat$, with the property that given $x\in X$ and $\alpha \in \supp
x$, we can find $n$ and $\alpha \in \Gamma_n$, so that
\[
\{\gamma \in \Gamma_n : |x(\gamma)|\geq |x(\alpha)|\}
\]
is finite. Then $X$ admits a pointwise lower
semicontinuous strictly convex norm or lattice norm, respectively.
\end{cor}

\begin{proof}
Let $P_n(x)(\gamma) = |x(\gamma)|$ whenever $\gamma \in \Gamma_n$
and $n \in \nat$. The coordinate maps are positive and convex. We
show that $P_n$ satisfies conditions (\ref{dis}) and (\ref{mid}) of
Theorem \ref{merc-suff}. Given distinct $x,y \in X$, take $n \in \nat$
and $\beta \in \Gamma_n$ such that $x(\beta) \neq y(\beta)$. Then
there is $\lambda \in (0,1)$ such that $\lambda x(\beta) +
(1-\lambda)y(\beta)$ is non-zero. Set $z = \lambda x + (1-
\lambda) y$ and take $n\in \nat$ such that
\[
A= \setcomp{\alpha \in \Gamma_n}{|z(\alpha)|\geq |z(\beta)|}
\]
is finite. Evidently $\beta \in A$, so $P_n(x)\restrict{A} \neq
P_n(y)\restrict{A}$, and
\[
|P_n(z)(\alpha)| \geq |z(\beta)| > |P_n(z)(\gamma)|
\]
whenever $\alpha \in A$ and $\gamma \in \Gamma_n \setminus A$.
\end{proof}

\begin{cor}[{\cite[Theorem 7]{smith:09}}]
If $K$ is Gruenhage then $\dual{C(K)}$ admits a strictly convex dual
lattice norm.
\end{cor}

\begin{proof}
If $K$ is Gruenhage then (cf.\ \cite[Lemma 6]{smith:09}), we can
find sequences $(\mathscr{U}_n)_{n=1}^\infty$ and
$(R_n)_{n=1}^\infty$ as in Definition \ref{gruenhage}, with the
further property that if $\mu \in \dual{C(K)}$ and $\mu(U) = 0$ for
all $U \in \mathscr{U}_n$, $n \in \nat$, then $\mu=0$. Let $\Gamma_n
= \mathscr{U}_n$ and define
\[
\Theta_n(\mu)(U) = |\mu|(U),\qquad U \in \mathscr{U}_n.
\]
Since $|\lambda\mu + (1-\lambda)\nu| \leq \lambda|\mu| +
(1-\lambda)|\nu|$ whenever $\lambda \in [0,1]$, the coordinate maps
$\Theta_n(\cdot)(U)$ are positive and convex. If $\mu,\nu \in
\dual{C(K)}$ are positive and distinct, then there exists $n \in
\nat$ and $U \in \mathscr{U}_n$ such that $\mu(U) \neq \nu(U)$. If
we set $\tau = \frac{1}{2}(\mu + \nu)$ then we have $\tau(U) >
\tau(R_n)$. By considering Definition \ref{gruenhage} part (2), we
see that for any $r > \tau(R_n)$, there are only finitely many $V
\in \mathscr{U}_n$ satisfying $\tau(V) \geq r$. Therefore,
conditions (\ref{dis}) and (\ref{mid}) of Theorem \ref{merc-suff}
apply to positive elements of $\dual{C(K)}$. Now we are able to
apply Theorem \ref{merc-suff}.
\end{proof}

Dashiell-Lindenstrauss spaces can be shown to
have strictly convex lattice norms in a similar way.

\section{Strictly convex dual norms on \dual{C(K)}}\label{norm}

Evidently, Theorem \ref{sccharac} relies on geometric assumptions,
in the sense that only sets having {\st} with {\em slices} are
considered. According to Example \ref{slicesnecessary}, it is not
always possible to remove the reliance on slices and deal instead
with open sets having no special geometric properties. However, we
can live without slices in an important special case. We devote this
section to proving the next result.

\begin{thm}
\label{scattered} Let $K$ be a scattered compact space. Then
$\dual{C(K)}$ admits a strictly convex dual (lattice)
norm if and only if $K$ has {\st}.
\end{thm}

Recall that any compact space $K$ embeds naturally
into $(\dual{C(K)},\weakstar)$ by identifying points $t \in K$ with
their Dirac measures $\delta_t$. It follows therefore from Theorem
\ref{scattered} that if $K$ is scattered and
$(\dual{C(K)},\weakstar)$ has {\st} (without slices), then
$(\dual{C(K)},\weakstar)$ has {\st} with slices. One implication of
Theorem \ref{scattered} may be proved easily.

\begin{prop}\label{nec}
If $\dual{C(K)}$ admits a strictly convex dual norm then $K$ has {\st}.
\end{prop}

\begin{proof}
By Theorem \ref{sccharac}, if $\dual{C(K)}$ admits a
strictly convex dual norm then $(\dual{C(K)},\weakstar)$ has {\st},
whence $K$ has {\st} by the natural embedding.
\end{proof}

In order to prove the converse implication, we need to refine our
{\st}-sequences so that they satisfy some additional properties.
Assume that a topological space $X$ admits a {\st}-sequence
$(\mathscr{U}_n)_{n=1}^\infty$. Given any finite sequence of natural
numbers $\sigma = (\oneton{n}{k})$, we define the family
\[
\mathscr{U}_\sigma \;=\; \setcomp{\bigcap_{i=1}^k U_i}{U_i \in
\mathscr{U}_{n_i}\mbox{ for all }i \leq k}.
\]
Let us also set $C_n = \bigcup \mathscr{U}_n$ and $C_\sigma =
\bigcup \mathscr{U}_\sigma$.

\begin{lem}
\label{finiteseparation} Assume that $F \subseteq X$ is a finite
subset such that for all $n$, either $F \cap C_n = \varnothing$ or
$F \subseteq C_n$. Then there exists $\sigma = (\oneton{n}{k})$ such
that $F \subseteq C_\sigma$ and, moreover, $F \cap V$ is at most a
singleton for all $V \in \mathscr{U}_\sigma$.
\end{lem}

\begin{proof}
Enumerate the set of doubletons $\{x,y\} \subseteq F$ as
$\{x_1,y_1\},\ldots,\{x_k,y_k\}$. For every $i$, there exists $n_i$
such that $\{x_i,y_i\} \cap C_{n_i}$ is non-empty and $\{x_i,y_i\}
\cap V$ is at most a singleton for all $V \in \mathscr{U}_{n_i}$. By
hypothesis, we have $F \subseteq C_{n_i}$ for all $i$. Put $\sigma =
(\oneton{n}{k})$. If $x \in F$, since $F \subseteq C_{n_i}$ for all
$i$, let $U_i \in \mathscr{U}_{n_i}$ so that $x \in
\bigcap_{i=1}^k U_i \in \mathscr{U}_\sigma$. Therefore $F \subseteq
C_\sigma$. Given $V = \bigcap_{i=1}^k V_i \in \mathscr{U}_{\sigma}$
and distinct $x,y \in F$, we have some $i$ such that $\{x,y\} \cap
W$ is at most a singleton for all $W \in \mathscr{U}_{n_i}$. In
particular, $\{x,y\} \cap V \subseteq \{x,y\} \cap V_i$ is at most a
singleton. This proves that $F \cap V$ is at most a singleton for
any $V \in \mathscr{U}_\sigma$.
\end{proof}

Bearing in mind the $\mathscr{U}_\sigma$, Lemma \ref{finiteseparation},
and by adding new singleton families if necessary, if $X$ has {\st}
then we can assume that there exists a {\st}-sequence with additional
properties, which we list in the next lemma.

\begin{lem}\label{newfamilies}
If $X$ has {\st} then it admits a {\st}-sequence $(\mathscr{U}_n)_{n=1}^\infty$ with the following properties.
\begin{enumerate}
\item $X = C_1$;
\item given $\oneton{n}{k} \in \nat$, there exists $m \in \nat$ such
that
\[
\mathscr{U}_m \;=\; \setcomp{\bigcap_{i=1}^k U_i}{U_i \in
\mathscr{U}_{n_i}\mbox{ for all }i \leq k};
\]
\item if $F$ is a finite subset of $X$ such that for each $n \in \nat$,
either $F \subseteq C_n$ or $F \cap C_n$ is empty, then there exists
$m \in \nat$ with two properties:
\begin{enumerate}
\item $F \subseteq C_m$;
\item $F \cap V$ is at most a singleton for all $V \in \mathscr{U}_m$.
\end{enumerate}
\end{enumerate}
\end{lem}

Armed with these enhanced {\st}-sequences, we can deliver the proof
of Theorem \ref{scattered}. We ask that our compact spaces be
scattered because the proof relies on the assumption that all
measures in $\dual{C(K)}$ are atomic.

\begin{proof}[Proof of Theorem \ref{scattered}.]
One implication was proved in Proposition \ref{nec}. Now assume that
$K$ is scattered and let $(\mathscr{U}_n)_{n=1}^\infty$ be a
{\st}-sequence for $K$ satisfying the properties of Lemma
\ref{newfamilies}. Given $n \geq 1$, $k \geq 0$ and finite $L \subseteq
\nat$, define the seminorm
\[
\pnorm{\mu}{n,k,L} = \sup\setcomp{|\mu|\left(\bigcup_{i \in L} C_i \cup \bigcup \mathscr{F} \right)}{\mathscr{F}\subseteq \mathscr{U}_n \mbox{ and }\card{\mathscr{F}} = k}.
\]
We show that these seminorms satisfy the requirements of Proposition
\ref{conv-charac}. To this end, suppose that $\mu$ and $\nu$ are positive,
and that
\begin{equation}\label{equality}
{\textstyle \pnorm{\mu}{n,k,L} = \pnorm{\nu}{n,k,L} = \frac{1}{2}\pnorm{\mu + \nu}{n,k,L}}
\end{equation}
for all $n$, $k$ and $L$. For a contradiction, we shall suppose also that
$\mu \neq \nu$. Since
\[
\pnorm{\mu}{1,0,\{n\}} = \mu(C_n)
\]
we have $\mu(C_n) = \nu(C_n) = \frac{1}{2}(\mu+\nu)(C_n)$ for all $n$, by (\ref{equality}). By
Lemma \ref{newfamilies} (1) and (2), and the inclusion-exclusion principle, if $I \subseteq \nat$ then
we know that
\[
\mu(C_{I,n}) = \nu(C_{I,n}) = {\textstyle \frac{1}{2}}(\mu+\nu)(C_{I,n})
\]
where
\begin{eqnarray*}
C_{I,n} &=& \bigcap_{i \leq n,i \in I} C_i \setminus \bigcup_{i \leq n, i \notin I} C_i.
\end{eqnarray*}
By monotone convergence, it follows that
\[
\mu(C_I) = \nu(C_I) = {\textstyle \frac{1}{2}}(\mu+\nu)(C_I)
\]
where
\[
C_I = \bigcap_{i \in I}C_i\setminus \bigcup_{i \in \nat\setminus I} C_i.
\]
Now $K$ is the disjoint union of the $C_I$, where $I$ ranges over non-empty subsets of $\nat$,
and since $\mu \neq \nu$ are atomic, we can find
non-empty $I \subseteq \nat$ such that $\mu\restrict{I}\, \neq \nu\restrict{I}$.
We fix this $I$ from now on. Take a countable set $A \subseteq C_I$ such that we can write
\[
\mu\restrict{I}\, = \sum_{t \in A} a_t\delta_t
\qquad\mbox{and}\qquad \nu\restrict{I}\, = \sum_{t \in A}
b_t\delta_t
\]
for some numbers $a_t,\,b_t \geq 0$. Let
\[
p = \max\setcomp{\max\{a_t,b_t\}}{t \in A, a_t \neq b_t}
\]
\[
q = \max(\setcomp{a_t}{a_t < p} \cup \setcomp{b_t}{b_t < p})
\]
and define the finite, possibly empty, set
\[
F = \setcomp{t \in A}{a_t = b_t \geq p}
\]
and let $k = \card{F}$. Take finite $G \subseteq A$ such that
\begin{equation}\label{Gest}
\sum_{t \in A\backslash G} a_t,\, \sum_{t \in A\backslash G } b_t <
{\textstyle\frac{1}{4}(p - q)}
\end{equation}
and $n$ large enough so that
\begin{equation}\label{nest}
\mu(C_{I,n}\setminus C_I),\, \nu(C_{I,n}\setminus C_I) <
{\textstyle\frac{1}{4}(p - q)}.
\end{equation}
Let $H = \{1,\ldots,n\} \cap I$ and $L = \{1,\ldots,n\} \setminus
I$. By Lemma \ref{newfamilies} (3), we can find $m\in \nat$ such
that $G \subseteq C_m$ and $G \cap V$ is at most a singleton for all
$V \in \mathscr{U}_m$. Since $C_I \subseteq \bigcap_{i \in H} C_i$,
we can and do assume that $C_m \subseteq \bigcap_{i \in H} C_i$, by
Lemma \ref{newfamilies} (2).

It is by considering the seminorm $\pnormdot{m,k+1,L}$ that we reach
our contradiction. Let $u \in A$ such that $a_u \neq b_u$ and
$\max\{a_u,b_u\} = p$. Clearly $u \notin F$. Also, $F \cup \{u\}
\subseteq G$. Indeed, if $t \in A\setminus G$ then $a_t, b_t <
\frac{1}{4}(p - q) < p$. Without loss of generality, assume that
$a_u < b_u = p$. Since $F \cup \{u\} \subseteq C_m$, it is possible
to find $\mathscr{G} \subseteq \mathscr{U}_m$ of cardinality $k+1$,
such that $F \cup \{u\} \subseteq \bigcup\mathscr{G}$.

By considering $\pnormdot{1,0,L}$ and (\ref{equality}), we know that
$\mu\left(\bigcup_{i \in L} C_i\right) = \nu\left(\bigcup_{i \in L}
C_i\right)$. We shall denote this common quantity by $c$. We estimate
\begin{eqnarray}
\label{nuest}\pnorm{\nu}{m,k+1,L} &\geq& \nu\left(\bigcup_{i \in L} C_i \cup \bigcup\mathscr{G} \right)\\
\nonumber &\geq& \nu\left(\bigcup_{i \in L} C_i\right) + \sum_{t \in F \cup \{u\}} b_t \qquad\mbox{as }(F \cup \{u\}) \cap \bigcup_{i \in L} C_i = \varnothing\\
\nonumber&\geq& c + p + \sum_{t \in F} b_t \;=\; c + p + \sum_{t \in F} a_t.
\end{eqnarray}
By (\ref{equality}) and the definition of the seminorms, let $\mathscr{H} \subseteq \mathscr{U}_m$
of cardinality $k+1$ be chosen in such a way that
\[
{\textstyle \frac{1}{2}}(\mu+\nu)\left(\bigcup_{i \in L} C_i \cup \bigcup\mathscr{H} \right) > \pnorm{\nu}{m,k+1,L} - {\textstyle \frac{1}{4}}(p-q).
\]

We claim that $a_t \geq p$ whenever $t \in \bigcup\mathscr{H} \cap
G$. In order to see this, first of all we claim that if $J \subseteq
A$ has cardinality at most $k$, then
\begin{equation}\label{Jineq}
\sum_{t \in J} a_t \leq \sum_{t \in F} a_t.
\end{equation}
Indeed, we have $\card{F\setminus J} \geq \card{J\setminus F}$,
since $\card{J} \leq k = \card{F}$. If $t \in J\setminus F$ then
either $a_t < p$ or $a_t \neq b_t$, which means $a_t \leq p$ by
maximality of $p$. Therefore
\begin{eqnarray*}
\sum_{t \in F} a_t - \sum_{t \in J} a_t &=& \sum_{t \in
F\setminus J} a_t - \sum_{t \in J\setminus F} a_t\\
&\geq& p(\card{F\setminus J}) - p(\card{J\setminus F}) \;\geq\; 0.
\end{eqnarray*}
This finishes the proof of the claim.

Now we can show that $a_t \geq p$ whenever $t \in \bigcup\mathscr{H}
\cap G$. If not, then $a_s < p$ for some $s \in \bigcup\mathscr{H}
\cap G$, meaning $a_s \leq q$. Observe that
\begin{equation}\label{subset}
\bigcup_{i \in L} C_i \cup \bigcup\mathscr{H} \subseteq \left( \bigcup\mathscr{H} \cap G\right) \cup \left(\bigcup\mathscr{H}
\cap C_I\setminus G\right) \cup \left( C_{I,n}\setminus C_I \right) \cup \bigcup_{i \in
L} C_i.
\end{equation}
To see this, it helps to note that
\[
\bigcup\mathscr{H}\setminus \bigcup_{i \in L} C_i \subseteq C_m \setminus \bigcup_{i \in L} C_i \subseteq \bigcap_{i \in H} C_i \setminus \bigcup_{i \in L} C_i = C_{I,n}.
\]
By choice of $m$, $\card \bigcup \mathscr{H} \cap G
\leq k+1$. Hence
\begin{eqnarray*}
\mu\left(\bigcup_{i \in L} C_i \cup \bigcup\mathscr{H} \right)
&\leq& \sum_{t \in F} a_t + a_s + {\textstyle\frac{1}{4}(p - q)} +
{\textstyle\frac{1}{4}(p - q)} + c \quad\mbox{by }(\ref{Gest}), (\ref{nest}), (\ref{Jineq}), (\ref{subset})\\
&\leq& \sum_{t \in F} a_t + q + {\textstyle \frac{1}{2}(p - q)} + c\quad\mbox{since }a_s \leq q\\
&\leq& \pnorm{\nu}{m,k+1,L} - {\textstyle \frac{1}{2}}(p-q) \qquad\mbox{by }(\ref{nuest}).
\end{eqnarray*}
However, this means
\[
{\textstyle \frac{1}{2}}(\mu+\nu)\left(\bigcup_{i \in L} C_i \cup \bigcup\mathscr{H} \right) \leq {\textstyle \frac{1}{2}}\pnorm{\nu}{m,k+1,L} - {\textstyle \frac{1}{4}}(p-q) + {\textstyle \frac{1}{2}}\pnorm{\nu}{m,k+1,L}
\]
which contradicts the choice of $\mathscr{H}$. Therefore $a_t \geq
p$ whenever $t \in \bigcup\mathscr{H} \cap G$. By a similar argument
applied to the $b_t$, we have $b_t \geq p$ whenever $t \in
\bigcup\mathscr{H} \cap G$. Hence we know that $a_t = b_t$ for $t
\in \bigcup\mathscr{H} \cap G$, lest we contradict the maximality of
$p$. It follows that $\bigcup\mathscr{H} \cap G \subseteq F$.
However, this forces
\begin{eqnarray*}
\mu\left(\bigcup_{i \in L} C_i \cup \bigcup\mathscr{H} \right) &\leq& \sum_{t \in F} a_t + {\textstyle\frac{1}{4}(p - q)} +
{\textstyle\frac{1}{4}(p - q)} + c \qquad\mbox{by }(\ref{Gest}), (\ref{nest})\mbox{ and } (\ref{subset})\\
&<& \pnorm{\nu}{m,k+1,L} - {\textstyle \frac{1}{2}}(p-q).
\end{eqnarray*}
Just as above, this contradicts the choice of $\mathscr{H}$.
\end{proof}

\begin{rem}\label{starproof}
Most of Theorem \ref{scattered} follows from Theorem \ref{sccharac}.
Starting with a {\st}-sequence from Lemma \ref{newfamilies}, we can
show directly that $(\dual{C(K)},\weakstar)$ has {\st} with slices.
For $n,k \in \nat$, finite $L \subseteq \nat$ and rational $q
> 0$, define $\mathscr{V}_{n,k,L,q,+}$ to be the family of all
$\weakstar$-open sets
\[
\setcomp{\mu \in \dual{C(K)}}{\mu_+\left(\bigcup_{i \in L} C_i \cup
\bigcup\mathscr{F}\right) > q}
\]
where $\mathscr{F} \subseteq \mathscr{U}_n$ has cardinality $k$.
Define $\mathscr{V}_{n,k,L,q,-}$ accordingly. By using essentially
the same method as that presented above, it can be shown that the
$\mathscr{V}_{n,k,L,q,\pm}$ form a {\st}-sequence. Moreover, if $V
\in \mathscr{V}_{n,k,L,q,\pm}$ then $\dual{C(K)}\setminus V$ is
convex. By the Hahn-Banach Theorem, each such $V$ can be written as
a union of $\weakstar$-open half-spaces. Therefore, we can write
down a {\st}-sequence for $(\dual{C(K)},\weakstar)$, the elements of
which being families of half-spaces. What we lose here is the fact
that the norm in Theorem \ref{scattered} is a lattice norm, which is
why we give the proof as is.
\end{rem}

\section{Topological properties of {\st} and examples}\label{topology}

In this section, we explore the properties of {\st} and see how it
compares with related concepts in the literature. In particular,
under the continuum hypothesis (CH), we provide an example of a
compact scattered non-Gruenhage space $K$ having {\st}. This means
that Theorem \ref{scattered} does not follow from existing results
such as Theorem \ref{smithgru}.

A topological space $X$ is said to have a {\em $\Gdelta$-diagonal}
if its diagonal
\[
\setcomp{(x,x)}{x \in X}
\]
is a $\Gdelta$ set in $X^2$. This concept has been studied
extensively in general metrisation theory; see, for example
\cite[Section 2]{gru:84}. It is easy to show  that $X$ has a
$\Gdelta$-diagonal if and only if the admits a sequence
$(\mathscr{G}_n)_{n=1}^\infty$ of open covers of $X$, such that
given $x,y \in X$, there exists $n$ with the property that $\{x,y\}
\cap U$ is at most a singleton for all $U \in \mathscr{G}_n$
\cite[Theorem 2.2]{gru:84}. Equivalently, if we consider the `stars'
\[
\str(x,n) = \bigcup\setcomp{U \in \mathscr{G}_n}{x \in U},
\]
then $\bigcap_{n=1}^\infty \str(x,n) = \{x\}$ for every $x \in X$.
In keeping with previous notation, we call such a sequence a {\em
$\Gdelta$-diagonal sequence}. Compact spaces with
$\Gdelta$-diagonals are metrisable (cf.\ \cite[Theorem
2.13]{gru:84}), so {\st} is evidently a strict generalisation of the
$\Gdelta$-diagonal property. In some cases, it is possible to reduce
problems about {\st} to the $\Gdelta$-diagonal case; see Theorem
\ref{ctblecpt} and Proposition \ref{o-spacenostar}, and also the
partitioning of $K$ into the $C_I$ in the proof of Theorem
\ref{scattered}.

Next, we compare {\st} with Gruenhage's property.

\begin{prop}\label{gruimpliesstar}
If $X$ is Gruenhage then it has {\st}.
\end{prop}

\begin{proof}
If $X$ is Gruenhage then let $(\mathscr{U}_n)_{n=1}^\infty$ and
$R_n$ be as in Definition \ref{gruenhage}. Let $\mathscr{V}_n =
\{R_n\}$ for each $n$. Given distinct $x,y \in X$, there exists $n$
and $U \in \mathscr{U}_n$, such that $\{x,y\} \cap U$ is a
singleton. If $x \in R_n$ then $y \notin R_n$ and it is true that
$\{x,y\} \cap U = \{x\}$ for every $U \in \mathscr{V}_n$, because
$\mathscr{V}_n$ is a singleton. Likewise if $y \in R_n$. So we
assume now that $x,y \notin R_n$. Now it is true that $\{x,y\} \cap
V$ is at most a singleton for every $V \in \mathscr{U}_n$, since if
$y \in V$ then $V \neq U$, and if $x \in V$ then $x \in U \cap V =
R_n$.
\end{proof}

There are an abundance of compact spaces which are Gruenhage, but
non-descriptive and so quite far from being metrisable; see
\cite[Corollary 17]{smith:09} or Theorem \ref{treestar} and
subsequent remarks, below. In Example \ref{starnon-gru}, we
show that under CH there exists a compact, scattered non-Gruenhage
space that has {\st}. Now we see that {\st} implies fragmentability.

\begin{prop}\label{stfrag}
If $X$ has {\st} then $X$ is fragmentable.
\end{prop}

\begin{proof}
Let $X$ have a {\st}-sequence $(\mathscr{U}_n)_{n=1}^\infty$. We
well order each $\mathscr{U}_n$ as $(U^n_\xi)_{\xi < \lambda_n}$.
Now define $V^n_\alpha = \bigcup_{\xi \leq \alpha} U^n_\xi$, for
$\alpha < \lambda_n$. We claim that given distinct $x,y \in X$,
there exists $n$ and $\alpha < \lambda_n$ such that $\{x,y\} \cap
V^n_\alpha$ is a singleton. As explained in the Introduction, this
is enough to give fragmentability. Indeed, take $n \in \nat$ with
the properties given in Definition \ref{star}, and pick the least
$\alpha < \lambda_n$ such that $\{x,y\} \cap U^n_\alpha$ is a
singleton. Then $\{x,y\} \cap U^n_\xi$ must be empty for all $\xi <
\alpha$, thus
$$
\{x,y\} \cap V^n_\alpha \;=\; \{x,y\} \cap U^n_\alpha
$$
is a singleton.
\end{proof}

Theorem \ref{ctblecpt} below is a generalisation of a result of
Chaber (cf.\ \cite[Theorem 2.14]{gru:84}), which states that
countably compact spaces with $\Gdelta$-diagonals are compact (and
thus metrisable). It allows us to glean a few more topological
consequences of the {\st} property. As preparation, fix an open
cover $\mathscr{V}$ of a countably compact (non-empty) space $X$.
Suppose that $X$ has a {\st}-sequence
$(\mathscr{U}_n)_{n=1}^\infty$, with $C_n = \bigcup \mathscr{U}_n$
for each $n$. Define
\[
\mathscr{A}_X = \setcomp{I \subseteq \nat}{X\setminus \bigcup_{n \in I}C_n \neq \varnothing}.
\]
Clearly, $\mathscr{A}_X$ is a hereditary family of subsets of
$\nat$. Moreover, it is compact in the pointwise topology. Indeed,
if $J \notin \mathscr{A}_X$, then by the countable compactness of
$X$, we can find finite $G \subseteq J$ such that $G \notin
\mathscr{A}_X$. It follows that
$\mathbb{P}(\nat)\setminus\mathscr{A}_X$ is open. Furthermore,
$\varnothing \in \mathscr{A}_X$ because $X$ is non-empty, so
$\mathscr{A}_X$ is also non-empty. From these facts, we deduce that
$\mathscr{A}_X$ admits an element that is maximal with respect to
inclusion.

\begin{thm}\label{ctblecpt}
If $X$ is countably compact and has {\st} then $X$ is compact.
\end{thm}

\begin{proof}
Fix an open cover $\mathscr{V}$ of $X$ and {\st}-sequence
$(\mathscr{U}_n)_{n=1}^\infty$ as above. We define a decreasing
transfinite sequence of countably compact subspaces $X_\alpha$ of
$X$, together with maximal $M_\alpha \in \mathscr{A}_{X_\alpha}$ and
finite $\mathscr{F}_\alpha \subseteq \mathscr{V}$, such that
\begin{enumerate}
\item $X_\alpha = X\setminus\bigcup_{\xi < \alpha}\bigcup\mathscr{F}_\xi$;
\item $M_\xi \notin \mathscr{A}_{X_\alpha}$ whenever $\xi < \alpha$.
\end{enumerate}
To begin, set $X_0 = X$. Given $X_\alpha$, we take some maximal
$M_\alpha \in \mathscr{A}_{X_\alpha}$ and set  $Y =
X_\alpha\setminus \bigcup_{n \in M_\alpha} C_n$. We claim that
$(\mathscr{U}_n)_{n \in \nat\setminus M_\alpha}$ is a
$\Gdelta$-diagonal sequence for $Y$. Indeed, the maximality of
$M_\alpha$ implies that $Y \subseteq C_n$ whenever $n \in
\nat\setminus M_\alpha$. If $x,y \in Y$ then by {\st}, there exists
$n$ such that $\{x,y\} \cap C_n$ is non-empty, and $\{x,y\} \cap U$
is at most a singleton for all $U \in \mathscr{U}_n$. By definition,
$Y \cap C_k$ is empty whenever $k \in M_\alpha$, so necessarily $n
\in \nat\setminus M_\alpha$. Our claim is proved.

By Chaber's result, $Y$ is compact. Therefore there exists a finite
set $\mathscr{F}_\alpha \subseteq \mathscr{V}$, such that
\[
X_\alpha\setminus\bigcup_{n \in M_\alpha} C_n = Y \subseteq
\bigcup\mathscr{F}_\alpha.
\]
Define $X_{\alpha+1} = X_\alpha^\prime =
X_\alpha\setminus\bigcup\mathscr{F}_\alpha$. We have (1) immediately
and (2) follows because $M_\alpha \notin \mathscr{A}_{X_{\alpha+1}}$
and $\mathscr{A}_{X_{\alpha+1}} \subseteq \mathscr{A}_{X_\alpha}$.
If $X_{\alpha+1}$ is empty then we stop the recursion. If $\lambda$
is a countable limit ordinal and $X_\alpha$ is non-empty for all
$\alpha < \lambda$, set $X_\lambda = \bigcap_{\alpha < \lambda}
X_\alpha$. (1) and (2) follow. By countable compactness, $X_\lambda$
is also non-empty.

This process has to stop at a countable (successor) stage, because
$(\mathscr{A}_{X_\alpha})$ is a strictly decreasing family of closed
subsets of the separable metric space $\mathbb{P}(\nat)$. Thus,
$X_{\alpha+1}$ is empty for some $\alpha < \wone$. By (1), we get
\[
X \subseteq \bigcup_{\xi \leq \alpha}\bigcup\mathscr{F}_\xi
\]
and so $X$ is covered by $\bigcup_{\xi \leq \alpha}\mathscr{F}_\xi$.
By a final application of countable compactness, we extract from
this a finite subcover.
\end{proof}

The next result generalises \cite[Corollary 4.3]{or:04} from
descriptive spaces to spaces with {\st}.

\begin{cor}\label{ctbletight}
If $L$ is locally compact and has {\st} then $L \cup \{\infty\}$ is countably tight and
sequentially closed subsets of $L \cup \{\infty\}$ are closed.
\end{cor}

\begin{proof}
The first assertion follows directly from Theorem \ref{ctblecpt} and
the second follows from Proposition \ref{stfrag} and the fact that
compact fragmentable spaces are sequentially compact (see
\cite[Corollary 2.7]{ribarska:87} and \cite[Lemma
2.1.1]{fabian:97}). Notice that if $L$ is any locally compact space
with {\st} then its 1-point compactification $L \cup \{\infty\}$ has
{\st} also. All we need to do is adjoin to any {\st}-sequence for
$L$ the singleton family $\{L\}$, which separates all points in $L$
from $\infty$.
\end{proof}

Concerning stability properties of {\st} under mappings, we have
the next result.

\begin{prop}\label{ctsimage}
If $K$ is a scattered compact space with {\st} and
$\mapping{\pi}{K}{M}$ is a continuous, surjective map then $M$ has
{\st}.
\end{prop}

\begin{proof}
If $K$ has {\st} then by Theorem \ref{scattered}, $\dual{C(K)}$
admits a strictly convex dual norm $\normdot$. If we define
$\mapping{T}{C(M)}{C(K)}$ by $T(f) = f \circ \pi$, it is standard to
check that
\[
\trinorm{\nu} = \inf\setcomp{\norm{\mu}}{\dual{T}(\mu) = \nu}
\]
defines a strictly convex dual norm on $\dual{C(M)}$.
Therefore $M$ has {\st}, again by Theorem \ref{scattered}.
\end{proof}

The proof above is concise and straightforward, but also utterly
opaque, as it leaves the reader with no idea of how to construct a
{\st}-sequence on $M$ in terms of a {\st}-sequence on $K$. We
outline a second approach to proving Proposition \ref{ctsimage},
which we include because we believe it gives the reader more idea of
what is going on. The dual map $S = \dual{T}$ above is a natural
extension of $\pi$ if we identify points in $K$ and $M$ with their
Dirac measures in $\dual{C(K)}$ and $\dual{C(M)}$, respectively. Set
\[
\Sigma = \setcomp{\mu \in \dual{C(K)}}{\mu \mbox{ is positive and
}\pnorm{\mu}{1}=1}.
\]
If $t \in M$ and $\mu \in \Sigma$ then $S(\mu) = t$ if and only if
$\supp \mu \subseteq \pi^{-1}(t)$. Given a {\st}-sequence
$(\mathscr{U}_n)_{n=1}^\infty$ on $K$ with the properties of Lemma
\ref{newfamilies}, together with the unions $C_n$, define the
$\weakstar$-compact and convex sets
\[
D_{n,q,L} = \setcomp{\mu \in \Sigma}{\mu\left(\bigcup_{i \in L} C_i
\cup U\right) \leq q\mbox{ for all }U \in \mathscr{U}_n}
\]
where $n \in \nat$, $q \in (0,1) \cap \rat$ and $L \subseteq \nat$
is finite. The $D_{n,q,L}$ should be compared to the seminorms
$\pnormdot{n,k,L}$ in the proof of Theorem \ref{scattered}. Given
distinct $s, t \in M$ and $\mu, \nu \in \Sigma$ in $S^{-1}(s)$ and
$S^{-1}(t)$ respectively, by following the proof of Theorem
\ref{scattered}, we can find $n$ and $q$ and $L$ such that
$\frac{1}{2}(\mu+\nu) \in D_{n,q,L}$, but $\{\mu,\nu\} \cap
D_{n,q,L}$ is at most a singleton. There is less to consider in this
case because as the supports of $\mu$ and $\nu$ are necessarily
disjoint, the set $F$ in the proof of Theorem \ref{scattered} is
empty. This is why we only need to consider individual elements of
$\mathscr{U}_n$ in the definition of the $D_{n,q,L}$, rather than
finite subsets of $\mathscr{U}_n$ as in the definition of the
$\pnormdot{n,k,L}$.

By appealing to compactness and convexity, it is possible to select
a finite set $G$ of triples $(n,q,L)$ with the property that if we
consider the intersection $D_G = \bigcup_{(n,q,L) \in G} D_{n,q,L}$,
then $D_G \cap S^{-1}(\frac{1}{2}(s+t))$ is non-empty, but either
$D_G \cap S^{-1}(s)$ is empty, or $D_G \cap S^{-1}(t)$ is empty.
Equivalently, $\frac{1}{2}(s+t) \in S(D_G)$, but $\{s,t\} \cap
S(D_G)$ is at most a singleton. The set $S(D_G)$ is
$\weakstar$-compact and convex, so the complement
$\dual{C(M)}\setminus S(D_G)$ can be written as the union of a
family $\mathscr{V}_G$ of $\weakstar$-open halfspaces of
$\dual{C(M)}$. From what we know, it can be easily verified that the
families $\mathscr{V}_G$, as $G$ ranges over all finite subsets of
triples $(n,q,L)$, induce a {\st}-sequence on $M$.

Now we move on to examples. We are chiefly interested in exploring
{\st}, Gruenhage's property and the gap between them. Given that
descriptive spaces are Gruenhage and spaces with {\st} are
fragmentable, we shall confine our attention to spaces that are
fragmentable but non-descriptive.

The first thing to point out is that {\st} is not equivalent to
fragmentability, because $\wone$ is scattered (hence fragmentable),
but does not have {\st}. That $\wone$ does not have {\st} is clear,
either directly from Corollary \ref{ctbletight}, or from Theorem
\ref{scattered} and \cite[Th\'eor\`eme 3]{tal:86}, which we
mentioned in the Introduction. Any locally compact space having
{\st} necessarily has a countably tight 1-point compactification,
but this condition is not sufficient. Hereafter, all of our examples
of locally compact spaces without {\st} have countably tight 1-point
compactifications.

Next, we consider trees. A {\em tree} $(T,\leq)$ is a partially
ordered set with the property that given any $t \in T$, its set of
predecessors $\setcomp{s \in T}{s \leq t}$ is well ordered. The tree
order induces a natural locally compact, scattered {\em interval
topology}. To render this topology Hausdorff, we shall only consider
trees $T$ with the property that every non-empty totally ordered
subset of $T$ has at most one minimal upper bound. An antichain is a
subset of $T$, no two distinct elements of which are comparable. For
further definitions and discussions about trees, and their role in
renorming theory, we refer the reader to
\cite{haydon:95,haydon:99,smith:06,smith:09,st:10,tod:84}.

If $P$ and $Q$ are partially ordered sets then we say that a map
$\mapping{\rho}{P}{Q}$ is {\em strictly increasing} if $\rho(x) <
\rho(y)$ whenever $x < y$. If such a map exists then we write $P
\preccurlyeq Q$. In \cite[Definition 5]{smith:06}, the second-named
author introduced a totally ordered set $Y$ to address the problem
of when $\dual{C_0(T)}$ admits a strictly convex dual norm. We
remark of $Y$ that $\real \preccurlyeq Y$, $Y^\alpha \preccurlyeq Y$
for all $\alpha < \wone$, where $Y^\alpha$ is ordered
lexicographically, and finally $Y$ contains no uncountable, well
ordered subsets \cite[Section 4]{smith:06}. By combining Theorem
\ref{scattered} with \cite[Corollary 17]{smith:09}, we obtain the
next result. See also \cite[Theorem 26]{st:10}.

\begin{thm}\label{treestar}
If $T$ is a tree then the following are equivalent.
\begin{enumerate}
\item $T$ is Gruenhage;
\item $T$ has {\st};
\item $\dual{C_0(T)}$ admits a strictly convex dual norm;
\item $T \preccurlyeq Y$.
\end{enumerate}
\end{thm}

Note that the 1-point compactification $T \cup \{\infty\}$ of a tree
$T$ is countably tight if and only if $T$ admits no uncountable
branches. Indeed, suppose that $T$ admits no uncountable branches.
Since each $t \in T$ admits a countable neighbourhood, the only
point we need to test is $\infty$.
If $\infty \in \closure{A}$ for some uncountable $A \subseteq T$,
then by a standard result of Ramsey theory, either $A$ contains an
uncountable totally ordered set or a countably infinite antichain
$E$. Only the second possibility is valid, whence $\infty \in
\closure{E}$. The converse implication follows immediately from the
fact that $\wone+1$ is not countably tight. Thus we restrict our
attention to trees with no uncountable branches.

Given a partially ordered set $P$, we set
\[
\sigma P = \setcomp{A \subseteq P}{A \mbox{ is well-ordered}}.
\]
Kurepa introduced this notion and proved the following fact: for all
$P$, we have $\sigma P \not\preccurlyeq P$. On the other hand, it is
straightforward to show that $\sigma\real^\alpha \preccurlyeq
\real^{\alpha} \times \{0,1\}$ \cite[Proposition 23]{smith:06}.
Moreover, it is known that $T$ is descriptive if and only if $T
\preccurlyeq \rat$ \cite[Theorem 4]{smith:06}. Therefore, we
conclude that $\sigma\rat$ and $\sigma\real^\alpha$, $\alpha <
\wone$, are all Gruenhage, non-descriptive spaces (see \cite[p.\
752]{smith:09} or \cite[p.\ 405]{st:10}). Instead, if we consider
any total order $W$ satisfying $Y \preccurlyeq W$, then $\sigma W
\not\preccurlyeq Y$ and so $\sigma W$ does not have {\st}. In
addition, if $W$ doesn't contain any uncountable well ordered
subsets, then $\sigma W$ is free of uncountable branches.

There is another type of tree without uncountable branches and
without {\st}. A subset $E$ of a tree is a {\em final part} if $u
\in E$ whenever $t \in E$ and $t \leq u$. If $E$ is a final part
then we say that $E$ is {\em dense} if every element of $T$ is
comparable with some element of $E$, and $T$ is called {\em Baire}
if every countable intersection of dense final parts (which is
itself a final part) is again dense. A subset $E$ is called {\em
ever-branching} if, given any $t \in E$, there exist incomparable
elements $u,v \in E$ satisfying $t < u,v$. If $T$ admits an
ever-branching Baire subtree then $C_0(T)$ does not admit a
G\^{a}teaux norm \cite[Theorem 2.1]{haydon:95}. Therefore, no such
tree can have {\st}. An ever-branching Baire tree without
uncountable branches exists; see \cite[Lemma 9.12]{tod:84} and
\cite[Proposition 3.1]{haydon:95}. Recall that a tree $T$ is called
{\em Suslin} if it contains no uncountable branches or antichains.
The existence of Suslin trees is independent of ZFC; see, for
example, \cite[Section 6]{tod:84}. Every Suslin tree contains an
ever-branching Baire subtree \cite[p.\ 246]{tod:84}, so we conclude
that no Suslin tree has {\st} either.

It is clear from Theorem \ref{treestar} that in order to find
examples of non-Gruenhage spaces with {\st}, we must search further
afield. A topological space $X$ is said to be {\em hereditarily
separable} (HS) if every subspace of $X$ is separable. Clearly, the
1-point compactification of a locally compact HS space is countably
tight. These spaces are interesting for us because if $K$ is
compact, HS and non-metrisable, then it is automatically
non-descriptive. This fact is stated in \cite[Proposition
4.2]{or:04} but no direct proof is given, so an argument is sketched
here for completeness. If $\mathscr{H}$ is an isolated family of
subsets of $K$ then $\mathscr{H}$ must be countable, because by
hereditary separability there is a countable subset of
$\bigcup\mathscr{H}$ which meets every member of $\mathscr{H}$.
Therefore, if $K$ is a descriptive compact HS space then it admits a
countable network, whence it is metrisable.

Since we want compact, non-metrisable HS spaces that are also
fragmentable, it is necessary to assume extra axioms. A space $X$ is
{\em hereditarily Lindel\"of} (HL) if every subspace of $X$ is
Lindel\"of. If $K$ is compact, fragmentable and HL then it is
metrisable (cf.\ \cite[Corollary 9]{kortezov:00}). Thus, we want HS
spaces that are not HL; such objects are called {\em S-spaces}. We
refer the reader to \cite{roitman:84} for an introduction to
$S$-spaces and also the related $L$-spaces. It is known that under
MA + $\neg$CH (where MA stands for Martin's axiom), there are no
compact $S$-spaces (cf.\ \cite[Theorem 6.4.1]{roitman:84}), and in
fact it is consistent that there are no $S$-spaces at all (cf.\
\cite[Theorem 7.2.1]{roitman:84}). Therefore, we must assume extra
axioms if we are to find any animals in this particular zoo.

Our treatment of $S$-spaces proceeds as follows. First, we outline
two approaches for constructing $S$-spaces by refining existing
topologies, and show that these yield Gruenhage spaces. Second, we
give an example under CH of a compact non-Gruenhage space of
cardinality $\aleph_1$ with {\st} and show that, given a further
mild assumption, no object of this kind can exist under MA +
$\neg$CH. Finally, we present a third method of constructing
$S$-spaces and show that no such space can have {\st}.

The spaces developed using the first approach are sometimes called
`Kunen lines', despite the fact that none of them are linearly
ordered. Assuming CH, the authors of \cite{jkr:76} develop a machine
which accepts as input a first countable HS space $(X,\rho)$ of
cardinality $\aleph_1$, and generates a finer topology $(X,\tau)$
which is locally compact, scattered, HS and non-Lindel\"of. In
applications, $X$ is usually a subset of $\real$ and $\rho$ is the
induced metric topology.

Later, this process was developed to ensure that $(X,\tau)^n$ is HS
for all $n \in \nat$; \cite[Section 7]{neg:84}. The resulting
1-point compactification $\mathcal{K}$ is known to Banach space
theorists as `Kunen's compact space'. It is not explicitly stated in
\cite[Section 7]{neg:84} that the resulting topology on $X$ refines
that of the real line, but the authors believe that it is meant to.
If the topology is such a refinement then necessarily the Euclidean
diameters of the $B^\alpha_k$ (which form the building blocks of
neighbourhoods of points, see (8) on \cite[p.\ 1124]{neg:84}) have
to tend to $0$ as $k \to \infty$. It can be checked that this
condition is also sufficient to produce a refinement. We note
further that an alternative approach to \cite[Section 7]{neg:84} is
given in \cite[Theorem 2.4]{vk:04}, and there, the fact that the
original topology is refined is explicitly stated.

Of course, it is clear that any refinement of a Gruenhage space is
again Gruenhage, because we can use exactly the same open sets to
separate points. Therefore, assuming the adjustment to the diameters
of the $B^\alpha_k$ above, we have the following result.

\begin{prop}\label{kunen-gru}
The Kunen lines are Gruenhage spaces. In particular,
$\dual{C(\mathcal{K})}$ admits a strictly convex dual norm, the
predual of which is necessarily G\^ateaux smooth.
\end{prop}

The second approach refines topologies as above, but this time using
the axiom $\mathfrak{b} = \aleph_1$, where $\mathfrak{b}$ is the
minimal cardinality of a subset of $\nat^\nat$ which is unbounded
with respect to the ordering of eventual dominance. Under
$\mathfrak{b} = \aleph_1$, it is shown in \cite[Theorem 2.5]{tod:89}
that the topology of any set of reals of cardinality $\aleph_1$ may
be refined to give a locally compact, scattered, non-Lindel\"of
topology which is HS in its finite powers.

\begin{prop}
The spaces of Todor\v{c}evi\'c in \cite[Theorem 2.5]{tod:89} are
Gruenhage.
\end{prop}

Before presenting our third approach to construct $S$-spaces, we
give our example under CH of a compact, scattered non-Gruenhage
space with {\st}. We shall adopt the same basic approach as
\cite{jkr:76} and \cite[Section 7]{neg:84}, and use an idea from
\cite{abd:88}. However, the underlying motivation for the space
should be compared, at a distance, to the split interval, rather
than the real line.

In fact, we construct a locally compact, scattered non-Gruenhage
space with a $\Gdelta$-diagonal. The 1-point compactification of
this space has {\st}. For our example, we shall make use of the
following observation about Gruenhage spaces of cardinality no
larger than the continuum.

\begin{prop}{\cite[Proposition 2]{st:10}}\label{oddity}
Let $X$ be a topological space with $\card{X} \leq \continuum$. Then
$X$ is Gruenhage if and only if there is a sequence
$(U_n)_{n=1}^\infty$ of open subsets of $X$ with the property that
if $x,y \in X$, then $\{x,y\} \cap U_n$ is a singleton for some $n$.
\end{prop}

In that which follows, $\sdiam$ denotes Euclidean diameter.

\begin{example}\label{starnon-gru}
(CH) There exists a locally compact, scattered, first countable
Hausdorff, non-Gruenhage space with a $\Gdelta$-diagonal.
\end{example}

\begin{proof}
Let $(x_\alpha)_{\alpha < \wone}$ be a set of distinct points in
$[0,1]$. Define $Y_\alpha = \setcomp{x_\xi}{\xi < \alpha}$ and
$X_\alpha = Y_\alpha \times \{\pm 1\}$ for $\alpha \leq \wone$, with
$Y = Y_{\wone}$ and $X = X_{\wone}$. Assuming CH, let
$(A_\alpha)_{\alpha < \wone}$ be an enumeration of all the countable
subsets of $Y$. Let $\mapping{t}{X}{X}$ be the map $t(x,i) =
(x,-i)$, and let $\mapping{q}{X}{Y}$ be the natural projection. We
obtain our topology on $X$ by building increasing topologies
$\tau_\alpha$ on the $X_\alpha$, $\alpha < \wone$, by transfinite
induction. The points $(x_\alpha,i)$, $i = \pm 1$, will have a
countable base of compact open neighbourhoods $U(x_\alpha,i,n)$, $n
\in \nat$, such that
\begin{enumerate}
\item if $\beta < \alpha$ then $X_\beta$ is open in $\tau_\alpha$ and
$\tau_\beta$ is the topology on $X_\beta$ induced by $\tau_\alpha$;
\item $U(x_\alpha,i,n) \setminus\{(x_\alpha,i)\} \subseteq
X_\alpha$;
\item $\diam{q(U(x_\alpha,i,n))} < 2^{-n}$;
\item $U(x_\alpha,-i,n) = t(U(x_\alpha,i,n))$;
\item $q\restrict{U(x_\alpha,i,n)}$ is injective;
\item if $\xi \leq \alpha$, $A_\xi \subseteq Y_\alpha$
and $x_\alpha \in \closure{A_\xi}^{\real}$, then
\[
U(x_\alpha,i,n) \cap (A_\xi \times \{-i\})
\]
is non-empty for every $n$.
\end{enumerate}

To take care of limit stages $\alpha$, we set
\[
\tau_\alpha = \setcomp{U \subseteq X_\alpha}{U \cap X_\beta \in
\tau_\beta \mbox{ for all }\beta < \alpha}.
\]
Now assume that $\tau_\alpha$ has been found. We construct
$\tau_{\alpha+1}$ by constructing neighbourhoods $U(x_\alpha,i,n)$, $n
\in \nat$, of the points $(x_\alpha,i)$, $i = \pm 1$.

If $x_\alpha \notin \closure{Y_\alpha}^\real$ then set
$U(x_\alpha,i,n) = \{(x_\alpha,i)\}$ for $i = \pm 1$ and $n \in \nat$.
Note that as $Y$ is a separable subset of $\real$, this can
happen for only countably many $\alpha$. Now assume that
$x_\alpha \in \closure{Y_\alpha}^\real$. Define
\[
F_\alpha = \setcomp{\xi \leq \alpha}{A_\xi \subseteq
Y_\alpha \mbox{ and }x_\alpha \in \closure{A_\xi}^{\real}}.
\]
Since $F_\alpha$ is at most countable, we can find an injective sequence
$(s_n)_{n=1}^\infty \subseteq Y_\alpha$ converging to $x_\alpha$, such that
\begin{enumerate}
\item[(i)] $\diam{\setcomp{s_m}{m \geq n}} < 2^{-n}$ for each $n$
\item[(ii)] $\setcomp{n \in \nat}{s_n \in A_\xi}$ is infinite
whenever $\xi \in F_\alpha$.
\end{enumerate}
By considering (3) applied to $\beta < \alpha$, and (i) above, for
every $n$ we can find $k_n$ such that
\begin{enumerate}
\item[(iii)] $q(U(s_n,-1,k_n)) \cap q(U(s_m,-1,k_m)) = \varnothing$
\end{enumerate}
whenever $n \neq m$, and
\begin{enumerate}
\item[(iv)] $\diam{q\left(\bigcup_{m \geq n} U(s_m,-1,k_m) \right)} < 2^{-n}$
\end{enumerate}
for every $n$. Finally, define
\[
U(x_\alpha,i,n) = \{(x_\alpha,i)\} \cup \bigcup_{m \geq n}
U(s_m,-i,k_m).
\]
These neighbourhoods are compact and open. Extend $\tau_\alpha$ to
$\tau_{\alpha+1}$ in the obvious way. It is clear that we have (1)
and (2), and then $\tau_{\alpha+1}$ is locally compact. (3) follows
from (iv) above. That $\tau_{\alpha+1}$ is Hausdorff follows by
inductive hypothesis, (3), and the fact that $U(x_\alpha,1,n) \cap U(x_\alpha,-1,m) =
\varnothing$. (4) and (5) follow from the inductive hypothesis,
the definition of $U(x_\alpha,i,n)$ and (iii) above. To see (6), note that
\[
\setcomp{s_m}{m \geq n} \times \{-i\} \subseteq U(x_\alpha,i,n) \cap (A_\xi \times \{-i\})
\]
so (6) now follows from (ii) above. This completes the induction. The
topology on $X$ is given by
\[
\setcomp{U \subseteq X}{U \cap X_\alpha \in
\tau_\alpha \mbox{ for all }\alpha < \wone}
\]
We show that $X$ is scattered. If $E \subseteq X$ is non-empty then
let $\alpha$ be minimal, such that $E \cap \{(x_\alpha,\pm 1)\}$ is
non-empty. If $(x_\alpha,i) \in E$ then by (1) and (2), we have that
$U = X_\alpha \cup \{(x_\alpha,i)\}$ is open, and $E \cap U =
\{(x_\alpha,i)\}$.

Now we show that $X$ has a $\Gdelta$-diagonal. Set
\[
\mathscr{G}_n = \setcomp{U(x,i,n)}{(x,i) \in X}.
\]
Let $(x,i),(y,j) \in X$. If $x \neq y$ then pick $n$ such that
$|x-y| \geq 2^{-n}$. We cannot have $(y,j) \in \str((x,i),n)$
because if so we would have $(x,i),(y,j) \in U(z,k,n)$ for some $(z,k)$,
giving
\[
|x-y| \leq \diam{q(U(z,k,n))} < 2^{-n}
\]
by (3). If $x = y$ and $i \neq j$ then by (5), we cannot have
$(x,i),(y,j) \in U(z,k,n)$ for any $(z,k)$ or $n$. Whatever the case,
\[
\bigcap_{n=1}^\infty \str((x,i),n) = \{(x,i)\}.
\]
This shows that $(\mathscr{G}_n)_{n=1}^\infty$ is a $\Gdelta$-diagonal sequence.

Finally, we prove that $X$ is not Gruenhage. Bearing in mind
Proposition \ref{oddity}, we
suppose for a contradiction that there exists a sequence of open subsets
$(V_n)_{n=1}^\infty$, with the property that given $(x,i),(y,j) \in X$,
we can find $n$ such that
\[
\{(x,i),(y,j)\} \cap V_n
\]
is a singleton. Define
\[
J_{n,i} = \setcomp{x \in Y}{(x,i) \in V_n \mbox{ and }(x,-i)
\notin V_n}.
\]
By assumption, $Y = \bigcup_{n,i} J_{n,i}$, so there
exist $n$ and $i$ such that $J = J_{n,i}$ is uncountable.
Remembering that $\real$ is HS, we can find a
countable subset $A_\xi$ such that $A_\xi \subseteq J \subseteq
\closure{A_\xi}^{\real}$. Because $J$ is uncountable, we can pick
$\alpha \geq \xi$ such that $A_\xi \subseteq Y_\alpha$ and
$x_\alpha \in J \subseteq \closure{A_\xi}^{\real}$.
Since $x_\alpha \in J$, we have $(x_\alpha,i) \in V_n$, so
take $m$ such that $U(x_\alpha,i,m) \subseteq V_n$. From (6), we
know that
\[
U(x_\alpha,i,m) \cap (A_\xi \times \{-i\}) \,\subseteq\, V_n \cap (J \times \{-i\})
\]
is non-empty. However, this violates the definition of $J$.
This contradiction establishes that $X$ is not Gruenhage.
\end{proof}

Together with Theorem \ref{scattered}, this example
shows that if $\dual{C(K)}$ admits a strictly convex dual norm then
$K$ is not necessarily Gruenhage. This gives a consistent negative solution
to \cite[Problem 14]{smith:09} and \cite[Problem 4]{st:10}.

We remark that the example above need not be HS. However,
it can easily be made to be HS by changing (ii) above to read
\begin{enumerate}
\item[(ii)] $\setcomp{n \in \nat}{s_{2n},s_{2n+1} \in A_\xi}$ is infinite
whenever $\xi \in F_\alpha$
\end{enumerate}
and setting
\[
U(x_\alpha,i,n) = \{(x_\alpha,i)\} \cup \bigcup_{m \geq n}
U(s_m,(-1)^m i,k_m).
\]
To see that this makes $X$ HS, we let $E \subseteq X$ and set $E_i =
\setcomp{x \in Y}{(x,i) \in E}$, $i = \pm 1$. Then take $\xi_i <
\wone$ such that $A_{\xi_i} \subseteq E_i \subseteq
\closure{A_{\xi_i}}^{\real}$ and choose $\alpha \geq
\xi_1,\,\xi_{-1}$ large enough to satisfy $A_{\xi_1} \cup
A_{\xi_{-1}} \subseteq Y_\alpha$. It can
now be verified that $E$ is in the closure of $E \cap X_\alpha$.

There is no hope of constructing something
like Example \ref{starnon-gru} in ZFC. A space $X$ is called {\em locally
countable} if every point of $X$ admits a countable neighbourhood. For example,
trees of height at most $\wone$ and `thin-tall' locally compact spaces are
locally countable. It is straightforward to see that a locally compact, locally
countable space must be scattered.

\begin{prop}[MA + $\neg$CH]\label{manotch}
Suppose that $L$ is a locally compact, locally countable space with {\st}
and $\card{L} < \continuum$. Then $L$ is $\sigma$-discrete.
\end{prop}

\begin{proof} This follows immediately from Corollary \ref{ctbletight}
and \cite[Theorem 2.1]{balogh:83}.
\end{proof}

It is not possible use stronger axioms to extend Proposition \ref{manotch}
to include spaces of cardinality $\continuum$: the tree $\sigma\rat$ is
locally compact, locally countable and Gruenhage, but is not $\sigma$-discrete.

We end this section by presenting our third class of $S$-spaces.
We shall call a regular, uncountable topological space $X$
an {\em $O$-space} if every open subset of $X$ is either countable
or co-countable. Ostaszewski constructed a locally compact, scattered
$O$-space using the clubsuit axiom $\clubsuit$ \cite[p.\ 506]{ost:76}.
It is known that $\clubsuit$ is independent of CH and that
$\clubsuit$ + CH is equivalent to Jensen's axiom $\Diamond$
(see \cite{shelah:80} and \cite[p.\ 506]{ost:76}, respectively).
It is possible to obtain $O$-spaces by assuming principles strictly
weaker than $\clubsuit$ \cite[Theorem 2.1]{juhasz:88}. Unlike the previous
constructions, these spaces are built from scratch, rather
than by refining an initial space.

Every $O$-space contains an $S$-subspace. Indeed, if $X$ is an
$O$-space then notice that at most one point of $X$ can fail to have
a countable open neighbourhood. Thus we can construct by induction
an uncountable subspace $Y = \setcomp{x_\alpha}{\alpha < \wone}$
such that $\setcomp{x_\xi}{\xi < \alpha}$ is open in $Y$ for every
$\alpha < \wone$. Thus $Y$ is not Lindel\"of. If, for a
contradiction, we suppose that $Z \subseteq Y$ is not separable,
then by another induction we can construct an uncountable,
relatively discrete subspace of $Y$. However, this cannot exist by
the $O$-space property. Therefore $Y$ is an $S$-space. We can argue
similarly to establish that every locally compact $O$-space has a
countably tight 1-point compactification.

\begin{prop}\label{o-spacenostar}
If $X$ is an $O$-space then it does not have {\st}.
\end{prop}

\begin{proof}
Suppose that $(\mathscr{U}_n)_{n=1}^\infty$ is a {\st}-sequence for
$X$, with $C_n = \bigcup\mathscr{U}_n$ for each $n$. Set
\[
J = \setcomp{n \in I}{C_n \mbox{ is uncountable}}.
\]
If $n \in J$ then $X\setminus C_n$ is countable, so
\[
E = \bigcup_{n \in J} \left( X\setminus C_n\right) \cup \bigcup_{n \in \nat\setminus J} C_n
\]
is also countable. If we let $A = X\setminus E$ then we see that
$A \subseteq C_n$ for all $n \in J$, and $A \cap C_n$ is empty
whenever $n \notin J$. For $x \in A$ and $n \in J$, define
\[
\str(x,n) = \bigcup\setcomp{U \in \mathscr{U}_n}{x \in U}.
\]
Since $(\mathscr{U}_n)_{n=1}^\infty$ is assumed to be a {\st}-sequence
for $X$, we have
\[
\{x\} = A \cap \bigcap_{n \in J} \str(x,n)
\]
for all $x \in A$, i.e.\ $(\mathscr{U}_n)_{n=1}^\infty$ induces a
$\Gdelta$-diagonal sequence on $A$. Given this, it follows that
for each $x \in A$, there exists some $n_x \in J$ such that
$\str(x,n_x)$ is countable. Indeed, otherwise,
\[
E \cup \bigcup_{n \in J} \left( X\setminus\str(x,n) \right)
\]
is countable, giving
\[
\{x\} = A \cap \bigcap_{n \in J} \str(x,n)
\]
uncountable. Since $A$ is uncountable, there exists $n$, which we
fix from now on, such that $B = \setcomp{x \in A}{n_x = n}$
is uncountable. Take an enumeration $(x_\alpha)_{\alpha < \wone}$ of
distinct points in $B$. We find $\alpha_0 < \alpha_1 < \alpha_2 <
\ldots < \wone$ such that
\[
x_{\alpha_\eta} \notin \bigcup_{\xi < \eta} \str(x_{\alpha_\xi},n)
\]
for all $\eta < \wone$. Observe that by the symmetry of the sets
$\str(x,n)$, we have $x_{\alpha_\xi} \notin \str(x_{\alpha_\eta},n)$
whenever $\xi \neq \eta$. Therefore $C = \setcomp{x_{\alpha_\xi}}{\xi < \wone}$
is a relatively discrete subspace, which is not permitted by the
$O$-space property.
\end{proof}

\begin{example}
Ostaszewski's space \cite[p.\ 506]{ost:76} is a locally compact, scattered
HS $O$-space. Therefore, it does not have {\st}.
\end{example}

By refining Ostaszewski's construction, it is possible to use
$\clubsuit$ to build a compact, scattered non-metrisable space $K$,
such that $K^n$ is HS for all $n$ \cite[Theorem 4.36]{hmvz:07}. Moreover,
it can be checked that this $K$ is, in addition, an $O$-space. Therefore,
unlike $\dual{C(\mathcal{K})}$, the space $\dual{C(K)}$ admits no strictly
convex dual norm.

We make a remark about this $C(K)$: the authors don't know if
it admits a G\^ateaux norm. Since $K$ is separable, $C(K)$ admits
a bounded linear, injective map into $\czero$. The authors don't know of
any example of an Asplund space with an injective map into a $\czerok{\Gamma}$,
which does not admit a G\^{a}teaux norm.

\section{Problems}\label{problems}

To finish, we present a number of related, unresolved problems. The first
problem stems from Theorem \ref{scattered}.

\begin{prob}\label{normprob}
If $K$ has {\st} and is not scattered, then does $\dual{C(K)}$ admit a
strictly convex dual norm?
\end{prob}

In fact, we don't even know if $\dual{C(L \cup \{\infty\})}$ admits a
strictly convex norm whenever $L$ is a locally compact space having a
$\Gdelta$-diagonal. The next problem is prompted by Example \ref{starnon-gru}.

\begin{prob}\label{zfcstarnon-gru}
Is there in ZFC an example of a non-Gruenhage compact space with {\st}?
\end{prob}

Proposition \ref{ctsimage} suggests the next problem.

\begin{prob}\label{ctsprob}
If $K$ has {\st} and is not scattered, and $\mapping{\pi}{K}{M}$ is
a continuous, surjective map, then does $M$ have {\st}? More
generally, if a topological space $X$ has {\st} and
$\mapping{f}{X}{Y}$ is a perfect, surjective map, does $Y$ have
{\st}?
\end{prob}

It is known that the answer to Problem \ref{ctsprob} is positive in
the Gruenhage case, including the more general perfect map assertion
\cite[Theorem 23]{smith:09}. It is also known that
$\Gdelta$-diagonals are not preserved under perfect images. In
\cite[Example 2]{burke:72}, an example is given of a locally compact
scattered space $L$ having a $\Gdelta$-diagonal, and a perfect
surjective map $\mapping{f}{L}{M}$, where the diagonal of $M$ is not a
$\Gdelta$. However, $\cderiv{L}{2}$ is empty, and the same
will apply to any perfect image of $L$, so all such images are
$\sigma$-discrete and therefore have {\st}. If Problem \ref{normprob}
has a positive solution then so will the first part of Problem
\ref{ctsprob}, simply by copying the proof of Proposition \ref{ctsimage}.

For our last problem, we refer the reader to the end of Section \ref{topology}.

\begin{prob}
Does $C(K)$ admit a G\^ateaux norm, where $K$ is the $O$-space of
\cite[p.\ 506]{ost:76} or \cite[Theorem 4.36]{hmvz:07}?
\end{prob}

%
%
%

\bibliographystyle{amsplain}

\end{document}